\documentclass{article}

\usepackage{graphicx}
\usepackage{amsmath}
\usepackage{natbib}
\usepackage{amsfonts}
\usepackage{amssymb}  
\usepackage{amsthm}  
\usepackage{subfigure}
\usepackage{verbatim}
\usepackage{color}
\usepackage{soul}
\usepackage{ulem}

\usepackage{mathrsfs,dsfont}
\RequirePackage[colorlinks,citecolor=blue,urlcolor=blue]{hyperref}

\newcommand{\R}{\mathbb{R}}
\newcommand{\N}{\mathbb{N}}
\newcommand{\I}{1{\hskip -2.5 pt}\hbox{I} }
\newcommand{\ddr}{\mathrm{d}}
\newcommand{\edr}{\mathrm{e}}
\newcommand{\Indic}{\mathds{1}}
\newcommand{\Small}[1]{\textstyle #1 \displaystyle}
\newcommand{\comillas}[1]{``\,#1\,''}
\newcommand{\F}{\mathcal{F}}
\newcommand{\Prob}{\mathrm{Pr}}
\newcommand{\diag}{\mbox{diag}}
\newcommand{\tr}{\mbox{tr}}
\newcommand{\bj}{\boldsymbol{j}}
\newcommand{\bl}{\boldsymbol{l}}
\newcommand{\sigmadown}{\underline{\sigma}}
\newcommand{\sigmaup}{\bar{\sigma}}

\newtheorem{theorem}{Theorem}
\newtheorem{corollary}{Corollary}
\newtheorem{proposition}{Proposition}
\newtheorem{lemma}{Lemma}
\newtheorem{remark}{Remark}

\begin{document}

\title{Posterior asymptotics of nonparametric location-scale mixtures for multivariate density estimation}

\author{Antonio Canale\thanks{antonio.canale@unito.it},
Pierpaolo De Blasi \thanks{pierpaolo.deblasi@unito.it}}
\date{Accepted version, \today}

\maketitle

\begin{abstract}
Density estimation represents one of the most successful applications of Bayesian nonparametrics. 
In particular, Dirichlet process mixtures of normals are the gold standard for density estimation and their asymptotic properties have been studied extensively, especially in the univariate case. 
However a gap between practitioners and the current theoretical literature is present. 
So far, posterior asymptotic results in the multivariate case are available only for location mixtures of Gaussian kernels with independent prior on the common covariance matrix, while in practice as well as from a conceptual point of view a location-scale mixture is often preferable. 
In this paper we address posterior consistency for such general mixture models by adapting a  convergence rate result which combines the usual low-entropy, high-mass sieve approach  with a suitable summability condition.
Specifically, we establish consistency for Dirichlet process mixtures of Gaussian kernels with various prior specifications on the covariance matrix.
Posterior convergence rates are also discussed.
\end{abstract}

{\center \textbf{Keywords: }}
Bayesian Nonparametrics; density estimation; Dirichlet mixture; factor model; posterior asymptotics; sparse random eigenmatrices

\section{Introduction}
\label{sec:intro}

Multivariate  density estimation is a fundamental problem in nonparametric inference being also the starting point for nonparametric regression, clustering, and robust estimation. For modeling continuous densities, standard nonparametric Bayes methods rely on Dirichlet process (DP) \citep{art:ferg:1973} mixture models of the form
\begin{equation}\label{eq:mix1}
  f(x) = \int K(x; \theta)dP(\theta),\quad P \sim DP(\alpha\,P^*),
\end{equation}
where $K(x; \theta)$ is a probability kernel depending on some finite-dimensional parameter $\theta$ and $DP(\alpha P^*)$ is a Dirichlet process with total mass $\alpha>0$ and $P^*$ a probability measure over the space of parameters $\theta$.  Model \eqref{eq:mix1} has been introduced by \citet{lo:1984} and made popular by \citet{esco:west:1995,mull:etal:1996}.
For  densities on $\R^d$, a common choice for the kernel $K(x;\theta)$
is the normal density $\phi_\Sigma(x-\mu)$, that is
  $$
  \phi_\Sigma(x-\mu)
  =(2\pi)^{-d/2} \det(\Sigma)^{-1/2} \exp\left\{-\frac{1}{2} 
  (x-\mu)^T \Sigma^{-1}(x-\mu) \right\},
  $$
for $\mu \in \R^d$ a $d$-dimensional vector of locations and $\Sigma$ a positive definite symmetric matrix of variance-covariances. Depending on whether or not the mixture involves the scale parameter $\Sigma$, we will speak of location and location-scale mixtures, respectively. In the former case, all components in the mixture share the same $\Sigma$ which is typically  modeled with an independent prior. 

In the univariate case, the asymptotic properties of DP mixtures are well known for both location  \citep{ghos:etal:1999,lijoi:pruen:walker:2005,ghos:vand:2007,walk:etal:2007} and location-scale mixtures  \citep{tokd:2006,ghos:vand:2001}.  
In the multivariate case, the only results, to our knowledge, are confined to the case of location mixtures. 
Posterior consistency is studied in \citet{wu:ghos:2010} assuming a truncated inverse-Wishart prior for $\Sigma$. \citet{shen:etal:2012} improve considerably on these results by deriving adaptive posterior convergence rates and remove the artificial truncation of the Wishart prior, which prevents an effective implementation. One of the main tools for achieving their result is the use of the stick-breaking representation of the DP to build a low-entropy, high-mass sieve on the space of mixed densities, a procedure whose extension to the multivariate case is a challenging task.

The lack of asymptotic results for mutivariate location-scale mixtures is somehow in contrast with their predominant use in applications \citep{mull:etal:1996,MacE:Mull:1998,gorur:2009, chen:etal:2010}.
In this paper, we fill this gap in the current literature by establishing posterior consistency for DP location-scale mixtures of multivariate normals with ready to verify conditions on the prior parameter $P^*$. In particular, a distinctive condition with respect to the location mixtures case involves the existence of moments to a certain order of the ratio between the largest and smallest eigenvalues of $\Sigma$, a quantity known in random matrix literature as condition number. This moment condition is satisfied by the inverse-Wishart as well as by other prior specifications that enable scaling to higher dimension. The consistency result exploits the sieve construction suggested by \citet{shen:etal:2012} by adapting a convergence rate theorem of \citet{ghos:vand:2007}. Such an adaptation allows to relax the growth condition on the entropy of the sieve through a summability condition which suitably weighs entropy numbers with prior probabilities, an idea first appeared in \citet{lijoi:pruen:walker:2005} and \citet{walk:etal:2007}. We also discuss some of the technical issues related to the challenging task of deriving posterior convergence rates for heavy tailed densities together with some preliminary results.

The layout of the paper is as follows. In Section~\ref{sec:cons}, sufficient conditions, on the true $f_0$ and on the prior, to obtain posterior  consistency of DP location-scale mixtures are given. In Section~\ref{sec:examples} some particular prior specifications satisfying such conditions are discussed.  Section~\ref{sec:rates} is about convergence rates and the paper ends with a final discussion.

\section{Posterior consistency}
\label{sec:cons}

For any $d \times d$ matrix $A$ with real eigenvalues, let $\lambda_1(A)\geq\ldots\geq\lambda_d(A)$ denote its eigenvalues in decreasing order and $\|A\|_2=\max_{x\neq 0}\|Ax\|/\|x\|$ be its spectral norm. We denote by $\mathcal{S}$ be the space of $d \times d$ positive definite matrices and by $\mathscr{P}$ the space of probability measures on $\R^d\times \mathcal{S}$. We consider DP location-scale mixtures of the type
\begin{equation}\label{eq:mixingDP2}
  f_P(x) = \int \phi_\Sigma(x-\mu) \ddr P(\mu, \Sigma), \,\,\,\,\, 
  P \sim DP(\alpha P^*).
\end{equation}
where $P^*\in\mathscr{P}$ with $\mu$ and $\Sigma$ independent under $P^*$.
In marginalizing out $P$ one can write model (\ref{eq:mixingDP2}) as 
\begin{equation}\label{eq:mixingDP2stick}
  f_{P}(x) = \sum_{h=1}^\infty \pi_h \phi_{\Sigma_h}(x-\mu_h), \quad
  (\mu_h,\Sigma_h) \sim \mbox{iid }P^*,  \quad
  \pi_h = V_h \prod_{k<h} (1-V_k),
\end{equation}
where  $V_h \sim \mbox{iid }\text{beta}(1,\alpha)$. 
We denote by $\Pi^*$ the DP prior on $\mathscr{P}$ and by $\Pi$ the prior induced by \eqref{eq:mixingDP2} on the space $\F$ of density functions on $\R^d$. 

As  customary in Bayesian asymptotics, we take the data $X_1\,\ldots,X_n$ to be i.i.d. from some \comillas{true} density $f_0\in{\cal F}$ and study the behavior of the posterior as $n\to\infty$ with respect to the $n$-products measure $F_0^n$, $F_0$ being the probability measure associated to $f_0$. 
 As metrics on $\F$ we consider the Hellinger  $d(f,g)=\{\int (\sqrt f-\sqrt g)^2\}^{1/2}$, and the $L_1$  $\|f-g\|_1=\int|f-g|$,  which induce equivalent  topologies in view of
  $d^2(f,g)\leq\|f-g\|_1\leq 2d(f,g)$.
By posterior consistency  at $f_0$ we mean that, for any $\epsilon >0$, 
  $$\Pi(\{f:\ \rho(f_0,f) > \epsilon\} \mid X_1, \dots, X	_n) \to 0$$ 
in $F_0^n$-probability where $\rho$ is either the Hellinger or the $L_1$-metric.

It is known that,  for the posterior distribution to accumulate around $f_0$, one has to establish first some support condition of the prior. We say that $\Pi$ satisfies the Kullback-Leibler (KL) property at $f_0$ if
\begin{equation}\label{eq:KL}
  \Pi\left\{f:\ \Small{\int}\log(f_0/f)f_0\leq \eta\right\}\geq 0,\quad
  \mbox{for any }\eta>0.
\end{equation}
The KL property has been established in Theorem 5 of \citet{art:wu:ghos:2008} for location-scale mixtures of Gaussian kernels with scalar covariance matrices, i.e. $\Sigma=\sigma^2 I$ for $\sigma^2>0$ and $I$ the $d\times d$ identity matrix. Minor adaptations are needed to extend this result to the case of non scalar $\Sigma$. 
Sufficient conditions for \eqref{eq:KL}  involve a mild requirement on the weak support of the DP prior $\Pi^*$ on $\mathscr{P}$ together with some regularity assumptions on $f_0$ like the existence of moments up to a certain order. As for the weak support of $\Pi^*$, we shall assume that the prior mean $P^*$ of the DP is supported on all $\R^d\times \mathcal{S}$, where we take as distance the sum of the Euclidean norm on $\R^d$ and the spectral norm on $\mathcal{S}$. All priors considered in Section 3 satisfy this requirement. As for $f_0$, the same regularity conditions of \citet[Theorem 5]{art:wu:ghos:2008} apply; they are repeated, for readers' convenience, in Lemma \ref{lemma:KLsupport} whose proof is reported in the Appendix.
%
%
\begin{lemma}\label{lemma:KLsupport}
Let $f_0 \in \mathcal{F}$ and $\Pi$ denote the prior on $\mathcal{F}$ induced by (\ref{eq:mixingDP2}). Assume that $f_0$ satisfies the following conditions: 
$0 < f_0(x) < M$ for some constant $M$ and all $x \in \R^d$; 
$|\int f_{0}(x) \log f_0 (x) \ddr x| < \infty$; 
for some $\delta >0$, $\int f_0(x) \log \frac{f_0(x)}{\phi_\delta(x)} \ddr x < \infty$, where $\phi_\delta(x) = \inf_{|| t - x||< \delta} f_0(t)$; 
for some $\eta >0$, $\int ||x||^{2(1+\eta)}f_0(x) \ddr x < \infty$.
Then $\Pi$ satisfies \eqref{eq:KL}.
\end{lemma}
The KL property \eqref{eq:KL} plays a very important role in consistency since it provides a lower bound for the denominator of the posterior probability. However, posterior consistency in non-compact spaces, such as $\cal F$, requires an additional condition on the prior which involves the metric entropy of $\cal F$ (see below for a formal definition).
A critical step is to introduce a compact subset $\mathcal{F}_n$, called sieve, which is indexed by the sample size $n$ and eventually grows to fill the entire parameter space as $n \to \infty$ . According to Theorem 2 of \citet{ghos:etal:1999}, the metric entropy of $\mathcal{F}_n$ has to grow slower than linearly in $n$, while the prior probability assigned to $\mathcal{F}_n^{c}$, the complement of the sieve,  needs to decrease exponentially fast in $n$. See Theorem 2.1 of \citet{ghos:etal:2000} for similar ideas applied to posterior convergence rates. 
The choice of the sieve is a delicate issue in multivariate density estimation, since the metric entropy tends to blow up with the dimension $d$. A novel sieve construction has been introduced in \citet{shen:etal:2012} and it has proven successful in deriving adaptive posterior convergence rates in the case of location mixtures with independent inverse-Wishart prior on $\Sigma$. In particular, the sieve relies on the stick-breaking representation of the DP as in \eqref{eq:mixingDP2stick}. In adapting this sieve construction to location-scale mixtures (see Lemma \ref{lemma:sieve} in the Appendix), the tail behavior of $P^*$ with respect to the condition number, i.e. the ratio of the maximum and minimum eigenvalue of $\Sigma$, plays a crucial role. A straight application of Theorem 2 of \citet{ghos:etal:1999} would require a too restrictive condition on this tail behavior, ruling out common choices for the part of $P^*$ involving $\Sigma$ like the inverse-Wishart distribution. See Remark~\ref{rem:nogoshetal} in the Appendix for a detailed explanation. For this reason we resort to a different posterior consistency  theorem which consists in a modification of Theorem 5 of \citet{ghos:vand:2007}. The main idea is to relax the growth condition on the entropy of $\F_n$ through a summability condition of entropy numbers weighted by square roots of prior probabilities. 
This modification can be applied also to relax the usual exponential tail behavior of the marginal of $P^*$ on the location parameters $\mu$, with a weaker power tail decay. As pointed out by \citet{ghos:vand:2007}, there is a trade off between entropy and summability which is worth exploring in Bayesian asymptotics. Our result is a step in this direction. Similar ideas had earlier appeared in \citet{lijoi:pruen:walker:2005} and in \citet{walk:etal:2007}.

To state the posterior convergence result, we recall the definition of entropy of ${\cal G}\subset{\cal F}$ as $\log N(\epsilon,{\cal G},d)$ where $N(\epsilon,{\cal G},d)$ is the minimum integer $N$ for which there exists $f_1,\ldots,f_N\in {\cal F}$ such that ${\cal G}\subset\bigcup_{j=1}^N\{f:\, d(f,f_j)<\epsilon\}$.
%
%
\begin{theorem}\label{th:Gho07}
Suppose ${\cal F}_n\subset \F$ can be partitioned as $\bigcup_j{\cal F}_{n,j}$ such that, for $\epsilon>0$,
\begin{gather}
  \label{eq:complement}
  \Pi(\F_n^c)\lesssim \edr^{-bn},\quad
  \mbox{for some }b>0\\
  \label{eq:summability1}
  \Small{ \sum_j \sqrt{N(2\epsilon,\F_{n,j},d)}\sqrt{\Pi(\F_{n,j})} }
  \edr^{-(4-c)n\epsilon^2}\to 0,\quad
  \mbox{for some }c>0
\end{gather}
Then $\Pi(f:\ d(f_0,f)>8\epsilon|X_1,\ldots,X_n)\to 0$ in 
$F_0^n$-probability for any $f_0$ satisfying \eqref{eq:KL}.
\end{theorem}
The proof is similar to that of Theorem 5 of \citet{ghos:vand:2007} and is presented in the Appendix. 
Here we state and prove the main result on posterior consistency for DP location-scale mixtures of Gaussian kernels.
%
%
\begin{theorem}\label{th:consistency}
Let $f_0$ satisfy the conditions of Lemma~\ref{lemma:KLsupport}. Consider the prior $\Pi$ defined in \eqref{eq:mixingDP2} with $P^*$ that satisfies the following tail behaviors: for some positive constants $c_1, c_2, c_3$, $r>(d-1)/2$ and $\kappa>d(d-1)$,
\begin{gather}
  \label{eq:1}
  P^*(\|\mu\|>x)\lesssim x^{-2(r+1)},\\
  \label{eq:2}
  P^*(\lambda_1(\Sigma^{-1})>x)\lesssim \exp(-c_1x^{c_2}),\\
  \label{eq:3}
  P^*(\lambda_d(\Sigma^{-1})<1/x)\lesssim x^{-c_3},\\
  \label{eq:4}
  P^*(\lambda_1(\Sigma^{-1})/\lambda_d(\Sigma^{-1})>x)
  \lesssim x^{-\kappa}
\end{gather}
for all sufficiently large $x>0$. Then the posterior is consistent at $f_0$.
\end{theorem}

Condition \eqref{eq:1} is weaker than the usual exponential tail condition in \citet{tokd:2006} (univariate case) and in \citet{shen:etal:2012} (location mixture case). 
The distinctive condition for multivariate location-scale mixtures turns out to be \eqref{eq:4} as explained in Remark~\ref{rem:nogoshetal}. 
It corresponds to the existence of the $d(d-1)+1$ moment of the quantity
  $\lambda_1(\Sigma^{-1})/\lambda_d(\Sigma^{-1})$
which is known in random matrix theory as condition number of $\Sigma^{-1}$. See \citet{edelman:sutton:2005}  for a recent contribution on the topic. In particular, the condition number plays an important role in electronical engineering problems in the context of multiple-input multiple-output communication systems. See \citet{matthaiou:etal:2010} for related results.
\begin{proof}[Proof of Theorem \ref{th:consistency}]
The proof is an application of Theorem \ref{th:Gho07} and is based on the entropy upper bounds of Lemma \ref{lemma:sieve} in the Appendix. Let $M_n=\underline\sigma_n^{-2c_2}=n$ and $H_n=\lfloor Cn\epsilon^2/\log n\rfloor$ for a positive constant $C$ to be determined later. Also, let 
$\bj=(j_1,\ldots,j_{H_n})$, $j_h\in\N^{*}$, and 
$\bl=(l_1,\ldots,l_{H_n})$, $l_h\in\N$. Define
\begin{align*}
  \F_n
  &=\{f_P\mbox{ with }P=\Small{\sum}_{h\geq 1}
  \pi_h\delta_{(\mu_h,\Sigma_h)}:\ 
  \Small{\sum}_{h>H_n}\pi_h\leq \epsilon;\\
  &\qquad\qquad\qquad\qquad\quad\;\;\, 
  \underline\sigma_n^2\leq\lambda_d(\Sigma_h),\lambda_1(\Sigma_h)\leq 
  \underline\sigma_n^2\left(1+\epsilon/\sqrt d\right)^{M_n},
  \mbox{for } h\leq H_n\}\\
  \F_{n,\bj,\bl}
  &=\left\{f_P\in\F_n:\
  \sqrt{n}(j_h-1)<\|\mu_h\|\leq \sqrt{n}j_h,\
  n^{2^{l_h-1}\Indic_{(l_h\geq 1)}} 
  <\Small{\frac{\lambda_1(\Sigma_h)}{\lambda_d(\Sigma_h)}}
  \leq n^{2^{l_h}}, 
  \mbox{ for }h\leq H_n\right\}
\end{align*}
so that $\F_n\uparrow \F$ as $n\to\infty$ and $\F_n\subset\bigcup_{\bj,\bl}\F_{n,\bj,\bl}$. 

As for \eqref{eq:complement} of Theorem \ref{th:Gho07}, note that
  $$\Pi(\F_n^c)
  \leq
  \Prob\left\{\Small{\sum}_{h>H_n}\pi_h>\epsilon\right\}
  +H_n\big[
  P^*(\lambda_d(\Sigma)<\underline\sigma_n^2)
  +P^*(\lambda_1(\Sigma)>\underline\sigma_n^2(1+\epsilon/\sqrt d)^{M_n})
  \big].$$  
Use the stick-breaking representation of the DP for the first term (see the proof of Proposition 2 of \citet{shen:etal:2012}), \eqref{eq:2} and \eqref{eq:3} to get
\begin{equation}\label{eq:prior_complement}
  \Pi(\F_n^c)\lesssim
  \left\{\frac{\edr\alpha}{H_n}\log\frac{1}{\epsilon}\right\}^{H_n}
  +H_n\left[\edr^{-c_1\underline\sigma_n^{-2c_2}}
  +\underline\sigma_n^{-2c_3}\left(1+\frac{\epsilon}
  {\sqrt d}\right)^{-c_3M_n}\right],
\end{equation}  
so that
\begin{align*}
  \Pi({\cal F}_n^c)
  &\lesssim (Cn\epsilon^2/\log n)^{-Cn\epsilon^2/\log n}
  +(Cn\epsilon^2/\log n)\left[\edr^{-c_1 n}
  +n^{c_3/c_2}(1+\epsilon/\sqrt d)^{-c_3n}\right]\\
  &\lesssim \exp\{-(Cn\epsilon^2/\log n)\log(Cn\epsilon^2/\log n)\}
  +\edr^{-c_1n}+\exp\{-c_3n\log(1+\epsilon/\sqrt d)\}.
\end{align*}
Note that $(Cn\epsilon^2/\log n)\log(Cn\epsilon^2/\log n)>Cn\epsilon^2/2$ for large enough $n$, therefore
  $$\Pi({\cal F}_n^c)
  \lesssim \edr^{-Cn\epsilon^2/2}+\edr^{-c_1n}
  +\edr^{-c_3\log(1+\epsilon/\sqrt d)n}
  \lesssim \edr^{-bn}$$
for $0<b<\min\{C\epsilon^2/2,c_1,c_3\log(1+\epsilon/\sqrt d)\}$. Hence \eqref{eq:complement} is satisfied.

We next show that $\F_{n,\bj,\bl}$ satisfies the summability condition \eqref{eq:summability1} of Theorem \ref{th:Gho07}. By an application of Lemma \eqref{lemma:sieve} in the Appendix,
\begin{multline*}
  N(\epsilon,\F_{n,\bj,\bl},\|\cdot\|_1)
  \lesssim\exp\left\{
  dH_n\log M_n+H_n\log\Small{\frac{C_1}{\epsilon}}
  \phantom{\left(
  \Small{\frac{\sqrt{n}j_h}{\underline\sigma_n\epsilon/2}+1}
  \right)^d}
  \right.\\
  \left.+\Small{\sum_{h\leq H_n}
  \log\left[\left(\frac{\sqrt{n}j_h}{\underline\sigma_n\epsilon/2}+1\right)^d
  -\left(\frac{\sqrt{n}(j_h-1)}{\underline\sigma_n\epsilon/2}-1\right)^d}\right]
  +\Small{\frac{d(d-1)}{2}}
  \log\Small{\frac{2d\,n^{2^{l_h}}}{\epsilon^2}}
  \right\}
\end{multline*}
for some positive $C_1$. Let $c_4=1/2+1/(2c_2)$ and note that
\begin{small}
\begin{align*}
\left(\frac{\sqrt{n}j_h}{\underline\sigma_n\epsilon/2}+1\right)^d
  -\left(\frac{\sqrt{n}(j_h-1)}{\underline\sigma_n\epsilon/2}-1\right)^d
  &=\left(\frac{2n^{c_4}j_h}{\epsilon}+1\right)^d
  -\left(\frac{2n^{c_4}j_h}{\epsilon}+1
  -\frac{2n^{c_4}}{\epsilon}-2\right)^d\\
  &\lesssim
  d\left(\frac{2n^{c_4}}{\epsilon}+2\right)
  \left(\frac{2n^{c_4}j_h}{\epsilon}+1\right)^{d-1}
  \lesssim \frac{n^{{c_4}d}j_h^{d-1}}{\epsilon^d}
\end{align*}
\end{small}
where inequality sign \comillas{$\lesssim$} is for both large $n$ and $j_h$. Hence we have
\begin{equation}\label{eq:entr_shell}
  \begin{split}
  N(\epsilon,\F_{n,\bj,\bl},\|\cdot\|_1)
  \lesssim \exp\Big\{
  dH_n\,\log n
  &+H_n\log\Small{\frac{C_1}{\epsilon}}\\
  &+\Small{\sum_{h\leq H_n}}
  \log\Small{\frac{n^{c_4d}j_h^{d-1}}{\epsilon^{d}}}
  +\Small{\frac{d(d-1)}{2}}
  \log\Small{\frac{2d\,n^{2^{l_h}}}{\epsilon^2}}\Big\}
  \end{split}
\end{equation}  
Moreover, by using tail conditions \eqref{eq:1} and \eqref{eq:4},
\begin{align}
  \Pi(\F_{n,\bj,\bl})
  &\leq\Small{\prod_{h\leq H_n}}
  P^*(\|\mu\|>\sqrt{n}(j_h-1),
  \lambda_1(\Sigma)/\lambda_d(\Sigma)>
  n^{2^{l_h-1}\Indic_{(l_h\geq 1)}})\nonumber\\
  \label{eq:prior_shell}
  &\lesssim\Small{\prod_{h\leq H_n}}
  [\sqrt{n}(j_h-1)]^{-\Indic_{(j_h\geq 2)}2(r+1)}
  n^{-\Indic_{(l_h\geq 1)}2^{l_h-1}\kappa}
\end{align}  
with the convention $0^0=1$.
A combination of \eqref{eq:entr_shell},  \eqref{eq:prior_shell} and 
$d^2(f,g)\leq\|f-g\|_1$ imply that 
  $\sqrt{N(2\epsilon,\F_{n,\bj,\bl},d)}
  \sqrt{\Pi(\F_{n,\bj,\bl})}$
is bounded by a multiple of
\begin{equation}\label{eq:bound}
  \Small{
  \exp\left\{\frac{d+c_4d}{2}Cn\epsilon^2\right\}
  \prod_{h\leq H_n}}
  j_h^{\frac{d-1}{2}}[\sqrt n(j_h-1)]^{-\Indic_{(j_h\geq 2)}(r+1)}\
  n^{\frac{d(d-1)}{4}2^{l_h}}
  n^{-\Indic_{(l_h\geq 1)}2^{l_h-1}\frac{\kappa}{2}
  }
\end{equation}  
By hypothesis, $r>(d-1)/2$ so that
  $K:=\sum_{j\geq2}j^{(d-1)/2}(j-1)^{-(r+1)}<\infty$. Hence by summing \eqref{eq:bound} with respect to $\bj$ we get
\begin{align}\label{eq:bound_2}
  &\Small{
  \exp\left\{\frac{d+c_4d}{2}Cn\epsilon^2\right\}
  (1+n^{-\frac{r+1}{2}}K)^{H_n}
  \prod_{h\leq H_n}}
  n^{\frac{d(d-1)}{4}2^{l_h}}
  n^{-\Indic_{(l_h\geq 1)}2^{l_h-1}\frac{\kappa}{2}
  }
\end{align}
Moreover, note that
\begin{align*}
  \sum_{l\geq0}
  n^{\frac{d(d-1)}{4}2^{l}}
  n^{-\Indic_{(l\geq 1)}2^{l-1}\frac{\kappa}{2}}
  &=n^{\frac{d(d-1)}{4}}
  +\sum_{l\geq 1}
  \exp\left\{-\frac{d(d-1)}{2}\log n
  \left(\frac{\kappa}{d(d-1)}-1\right)2^{l-1}\right\}
  \\
  &\leq 2n^{\frac{d(d-1)}{4}}
\end{align*}
for $n$ large enough since, by hypothesis, $\kappa/[d(d-1)]-1>0$. 
Hence, by summing \eqref{eq:bound_2} with respect to $\bl$ we get the upper bound
\begin{align*}
  &\exp\left\{\frac{d+c_4d}{2}Cn\epsilon^2\right\}
  (1+n^{-\frac{r+1}{2}}K)^{H_n}
  \left[2n^{\frac{d(d-1)}{4}}
  \right]^{H_n}\\
  &\leq
  \exp\left\{\frac{d+c_4d}{2}Cn\epsilon^2\right\}
  4^{H_n} 
  \exp\left\{Cn\epsilon^2\frac{d(d-1)}{4}\right\}
  \lesssim
  \exp\left\{\frac{1}{2}\left[d+c_4d+\frac{d(d-1)}{2}\right]
  C\epsilon^2n\right\}
\end{align*}  
where in the first inequality we have used $n^{-\frac{r+1}{2}}K\leq 1$ for $n$ large enough. By taking $C$ sufficiently small to satisfy
  $C<2(4-c)/[d+c_4d+d(d-1)/2]$
for some $c>0$, \eqref{eq:summability1} is satisfied and the proof is complete.
\end{proof}

\section{Illustration}
\label{sec:examples}

Theorem \ref{th:consistency} holds for DP location-scale mixtures of multivariate Gaussian kernels with minimal requirements on the prior parameter $P^*$. The power tail decay \eqref{eq:1} for $\mu$ is indeed important as it covers the prior specification of \citet{mull:etal:1996}:
  \[
  P^*(\mu,\Sigma)=N(\mu; m, B)\, IW(\Sigma;\Sigma_0,\nu),
  \]
where an additional hyperprior on $B$ is given by $IW(B; B_0,\nu_B)$, with $IW(B_0,\nu_B)$ denoting the inverse-Wishart distribution with scale parameter $B_0$ and $\nu_B$ degrees of freedom. In this case $\mu$ under $P^*$ has multivariate Student's $t$-distribution with $\nu_B$ degree of freedom so that $P^*(\|\mu\|^2>x)\sim x^{-(\nu_B-d+1)/2}$  for $x$ large enough. Hence $\nu_B>2d$ is sufficient for \eqref{eq:1} to hold with $r>(d-1)/2$.

In what follows, we focus  on the prior specification of the scale parameter $\Sigma$ and show that conditions \eqref{eq:2}-\eqref{eq:4} are verified for some important choices which make a substantial practical difference in applications, particularly in high-dimensional settings.  
Henceforth, we refer to the prior specification for $\Sigma$ as $\Sigma\sim \mathcal{L}$ meaning that $\mathcal L$ is the marginal of $\Sigma$ with respect to the prior mean $P^*$. All proofs are reported in the Appendix.

The first result is about $\mathcal L$ being the inverse-Wishart distribution $IW(\Sigma_0,\nu)$. While conditions \eqref{eq:2} and \eqref{eq:3} are always satisfied, see e.g. the proof of Lemma 1 of \citet{shen:etal:2012}, $\nu$ needs to be sufficiently large for tail condition \eqref{eq:4} on the condition number to hold. Corollary \ref{cor:consistencywhishart} makes this statement precise.
%
%
\begin{corollary}
Assume $f_0$ satisfies the conditions of Lemma~\ref{lemma:KLsupport}. 
Consider a prior $\Pi$ induced by \eqref{eq:mixingDP2} with $P^*$ satisfying \eqref{eq:1} and $\mathcal{L} = IW(\Sigma_0,\nu)$ with $\nu > 2d(d-1)+d-1$. 
Then conditions \eqref{eq:2}-\eqref{eq:4} are satisfied and the posterior is consistent at $f_0$.
\label{cor:consistencywhishart}
\end{corollary}

There is a rich literature providing alternatives to the inverse-Wishart prior when the dimension of the data is large. For example a commonly used and successful approach consists on analytic factorizations \citep{west:2003,carv:etal:2008} where
\begin{equation}\label{eq:Sigmaprior}
  \Sigma = \Gamma \Gamma^T + \Omega, \ 
  \Gamma \sim \mathcal{L}_\Gamma, \; \Omega \sim \mathcal{L}_\Omega, 
\end{equation}
where $\Gamma$ is a $d \times r $ matrix with $r < d$, independent from the $d \times d$ diagonal matrix $\Omega$. 
Let $\gamma_{jh}$ be the $(j,h)$-th element of $\Gamma$ (factor loading) and $\sigma_j^2$ be the $j$-th diagonal element of $\Omega$ (residual variance). The specification of $\mathcal{L}_{\Gamma}$ and $\mathcal{L}_{\Omega}$ corresponds then to a distribution for $\gamma_{jh}$ and $\sigma^2_j$. 
The next corollary addresses the case of normal factor loadings with inverse gamma distributed residual variances. It turns out that conditions \eqref{eq:2}-\eqref{eq:3} are automatically satisfied, while a constraint on the shape parameter of the inverse gamma prior is needed for the verification of condition \eqref{eq:4}.
%
%
\begin{corollary}
Assume $f_0$ satisfies the conditions of Lemma~\ref{lemma:KLsupport}. 
Consider a prior $\Pi$ induced by \eqref{eq:mixingDP2} with $P^*$ satisfying \eqref{eq:1} and  $\mathcal{L}$ induced by \eqref{eq:Sigmaprior}. 
Assume that  $\gamma_{ij} \sim \mbox{iid }N(0,1$) and $\sigma_j^{-2} \sim\mbox{iid } \text{Ga}(a,b)$ with $a>d(d-1)$. 
Then conditions \eqref{eq:2}-\eqref{eq:4} are satisfied and the posterior is consistent at $f_0$.
\label{cor:consistencyfactor}
\end{corollary}
%
%
\begin{remark}\label{remarkanirban}
Motivated by the need of a method that scales for increasing dimension, \citet{bhat:duns:2011} introduce a sparse Bayesian factor model for the estimation of high-dimensional covariance matrices according to \eqref{eq:Sigmaprior}. The model  consists of a multiplicative gamma process shrinkage prior on the factor loadings, i.e.
\begin{align*}
	& \gamma_{jh}|\phi_{jh} \tau_h \sim N(0,\phi_{jh}^{-1} \tau_h^{-1}), \,\,
	\phi_{jh}\sim\text{Ga}(3/2,3/2), \,\, \tau_h = \prod_{l=1}^h \delta_l,\\
	& \delta_1 \sim \text{Ga}(a_1,1), \,\, 
	\delta_l \sim \text{Ga}(a_2,1), l>1, \,\, 
	\sigma_j^{-2} \sim \text{Ga}(a,b), j = 1, \dots, d.
\end{align*}
Using similar arguments to Corollary~\ref{cor:consistencyfactor} it can be proved that conditions \eqref{eq:2}--\eqref{eq:4} are satisfied for $a>d(d-1)$.
\end{remark}
One can easily build $\cal L$ which satisfy \eqref{eq:2}--\eqref{eq:4} by modeling directly the distribution of the eigenvalues. The idea is to use  the spectral decomposition of $\Sigma$:
\begin{equation}\label{eq:spectralprior} 
  \Sigma= O \Lambda O^T, \ 
  O \sim \mathcal{L}_O, \; \Lambda \sim \mathcal{L}_\Lambda, \; 
\end{equation}  
where $O$ is a $d\times d$ orthogonal matrix independent from  $\Lambda=\diag(\lambda_1, \dots, \lambda_d)$ with $\lambda_i>0$.
It is clear that the verification of \eqref{eq:2}--\eqref{eq:4} involves only the distribution $\mathcal{L}_\Lambda$.
Motivated by high dimensional sparse random matrices modeling, \citet{cron:west:2012} propose a similar approach. The authors set $O_{i,j}$ to be the rotation matrix for the rotator angle $\omega_{i,j}$ and $O = \prod_{i<j} O_{i,j}(\omega_{i,j})$ and define prior on $O$ throught a prior on the rotator angles $\omega$ which naturally accommodates sparsity, namely
\begin{align*}
	& p(\omega) = \beta_{\pi/2}\I(\omega = \pi/2) + 
	(1-\beta_{\pi/2})\beta_0 \I(\omega=0) + 
	(1-\beta_{\pi/2})(1-\beta_0) p_c(\omega),\\
	& p_c(\omega) = c(\kappa) \exp\{\kappa \cos^2\omega\}
	\I(|\omega| < \pi/2).  
\end{align*}
This class of models is particularly useful to induce sparsity without the assumption of a reduced dimensional latent factor and  hence can be appealing in many practical situations. The next corollary discusses sufficient conditions on $\mathcal{L}_\Lambda$ to obtain posterior consistency under formulation \eqref{eq:spectralprior}.
%
%
\begin{corollary}
Assume $f_0$ satisfies the conditions of Lemma~\ref{lemma:KLsupport}. 
Consider a prior $\Pi$ induced by \eqref{eq:mixingDP2} with $P^*$ satisfying \eqref{eq:1} and  $\mathcal{L}$ following  \eqref{eq:spectralprior} with $\lambda_i^{-1} \sim \mbox{Ga}(a,b)$, $a > d(d-1)$. 
Then conditions \eqref{eq:2}-\eqref{eq:4} are satisfied and the posterior is consistent at $f_0$.
\label{remarkandrew}
\end{corollary}


\section{Posterior convergence rates}
\label{sec:rates}

In this section we discuss some relevant issues related to the derivation of posterior convergence rates for location-scale mixtures. 
Under the prior conditions \eqref{eq:1}--\eqref{eq:4}, the sieve construction laid down in the proof of Theorem 2 would adapt to any rate $\epsilon_n=n^{-\gamma}$ (for any $\gamma\in(0,1/2)$ and up to a logarithmic term) determined by the prior concentration rate. 
This statement is made precise in the following proposition, whose proof is reported in the Appendix. 
%
%
\begin{proposition}\label{prop:1}
Let $\tilde\epsilon_n=n^{-\gamma}(\log n)^{t}$ for some $\gamma\in(0,1/2)$ and $t\geq0$ and suppose that 
\begin{equation}\label{eq:prior_KLrate}
  \Pi\left(
  f:\ \int f_0\log(f_0/f)\leq\tilde\epsilon_n^2,\,
  \int f_0\log^2(f_0/f)\leq\tilde\epsilon_n^2
  \right)\geq \edr^{-n\tilde\epsilon_n^2}
\end{equation}  
Then, under the hypothesis of Theorem \ref{th:consistency}, 
  $\Pi_n\{f: d(f_0,f)>(\log n)^s\tilde\epsilon_n\}\to 0$ 
in $F_0^n$-probability for any $s>0$.
\end{proposition} 
As for the prior contraction rate \eqref{eq:prior_KLrate}, in the case of location mixtures the derivation of adaptive $\tilde\epsilon_n$ for $\beta$-H\"older $f_0$ consists of three main steps: (i) construct a density $h_\sigma$ (depending on $f_0$ and $\beta$) such that 
  $d(f_0,\phi_{\sigma}\star h_\sigma)\lesssim\sigma^\beta$
as $\sigma\to 0$; (ii) modify $h_\sigma$ to a compactly supported density $\tilde h_\sigma$ with the same approximation properties;
(iii) approximate the continuous mixture $\phi_{\sigma}\star \tilde h_\sigma$ by a discrete mixture $\phi_{\sigma}\star F_\sigma$ with a convenient lower bound of the ratio $\phi_{\sigma}\star F_\sigma/f_0$. Here $\phi_{\sigma}\star h$ denotes the convolution of $h$ and $\phi_{\sigma^2 I}$. See \citet{shen:etal:2012} and \citet{krui:2010}. 
Steps (ii) and (iii) rely on the assumption of exponential tail of $f_0$, which makes these techniques not suitable for location-scale mixtures,
since the latter are more flexible than location mixtures in modeling densities with heavy tails. 
We come back to this point later. Before, we present a result of the type of convergence rates which are attainable by adapting these techniques to our setting. To this aim, refer to the definition of the locally $\beta$-H\"older class with envelope $L$, denoted ${\cal C}^{\beta,L,\tau_0}(\R^d)$, from \citet{shen:etal:2012}. Moreover, for a multi-index $k=(k_1,\ldots,k_d)$, $k_i\in\N^*$, define $k_.=k_1+\cdots+k_d$ and let $D^k$ denote the mixed partial derivative operator $\partial^{k_.}/\partial^{k_1}\cdots\partial^{k_d}$.
Finally, we make the following two assumptions on the prior mean $P^*$. As for the scale parameter, we resort to an assumption analogous to (4) in \citet{shen:etal:2012} about the mass in neighborhood of the eigenvalues of $\Sigma$: there exist $\kappa^*,a_4,a_5,b_4,C_3$ such that for any $s_1\geq \cdots\geq s_d\geq 0$ and $t\in(0,1)$,
\begin{equation}\label{eq:Gstar}
  P^*\left(s_i<\lambda_i(\Sigma^{-1})<s_i(1+t),\
  i=1,\ldots,d\right)\gtrsim b_4s_d^{a_4}t^{a_5}\exp\{-C_3s_1^{\kappa^*/2}\}.
\end{equation} 
As for the mean parameter, we assume that the marginal of $P^*$ has density $f^*$ and that $f^*$ has tails no lighter than the Gaussian, i.e. for $b^*,\tau^*>0$ and $\|x\|$ sufficiently large,
\begin{equation}\label{eq:fstar}
  f^*(x)\gtrsim \edr^{-b^*\|x\|^{\tau^*}}
\end{equation} 
The proof of the following proposition is deferred to the Appendix. 
%
%
\begin{proposition}\label{prop:2}
Let $f_0\in{\cal C}^{\beta,L,\tau_0}(\R^d)$ be a bounded probability density function satisfying 
  $F_0(|D^k f_0|/f_0)^{(2\beta+\epsilon)/k_.}<\infty$,
$k_.\leq\lfloor\beta\rfloor$, 
  $F_0(L/f_0)^{(2\beta+\epsilon)/\beta}<\infty$
for some $\epsilon>0$ and
\begin{equation}\label{eq:8}
  f_0(x)\leq c\exp\{-b\|x\|^\tau\},
\end{equation} 
for some $b,c>0$, $\tau\geq2$ and $\|x\|$ sufficiently large. 
Assume that conditions \eqref{eq:Gstar} and \eqref{eq:fstar} hold for $\tau^*\leq \tau$. Then \eqref{eq:prior_KLrate} holds for
\begin{equation}\label{eq:rate_sub}
  \tilde\epsilon_n=n^{-\beta/(2\beta+d+\kappa^*)}(\log n)^t,\quad
  t\geq \frac{d(1+(\kappa^*+1)/\beta+1/\tau)}{2+(d+\kappa^*)/\beta)}
\end{equation}
\end{proposition}
Proposition \ref{prop:2} yields, via Proposition \ref{prop:1}, a posterior convergence rate which is suboptimal with respect to the minimax rate $n^{-\beta/(2\beta+d)}$ by a term which depends on the constant $\kappa^*$ appearing in \eqref{eq:Gstar}. For illustration, $\kappa^*=2$ if, under $P^*$, $\Sigma$ has Inverse Wishart distribution $\mbox{IW}(\Psi,\nu)$ with $\nu$ degrees of freedom and a positive definite scalar matrix $\Psi$, see Lemma 1 in \citet{shen:etal:2012}. Another useful specification is to consider that, under $P^*$, each $\lambda_i(\Sigma^{-1})$ have been independently assigned the distribution of the square of an inverse gamma random variable. Then $\kappa^*=1$, which leads to a better convergence rate, still suboptimal. 
The technical reason for which the minimax rate is not achieved is to be found in the prior probability of $L_1$-balls around the frequencies of the approximating mixture of the true density. 
Specifically, the approximating mixture has all covariance matrices equal to $\sigma_n^2 I$ for a scaling factor $\sigma_n$ which goes to zero as a function of $\tilde \epsilon_n$. In location-scale mixtures $P^*$ put mass proportional to $\exp\{-C_3\sigma_n^{-\kappa^*}\}$ in neighborhoods of $\sigma_n^2 I$ and this accounts for an extra factor in the radius of the $L_1$-ball with respect to the location mixture case. This, in turns, determines worse probability estimates of Kullback-Leibler neighborhoods. See the proof of Proposition \ref{prop:2} for details.
It can be proved that the minimax rate $n^{-\beta/(2\beta+d)}$ is recovered upon setting the marginal of $P^*$ on $\Sigma$ depending on $n$ by the scaling factor equal to $n^{-1/(2\beta+d)}(\log n)^{-1/\beta}$, however this has clearly a limited relevance since the rate would be non adaptive and it would rule out the Inverse-Wishart distribution as well as any other commonly used prior specification for the covariance matrix.
The suboptimal convergence rate might be related to the fact that location-scale mixtures are more robust to tails than location mixtures. This should emerge through posterior convergence rates which are slower still adaptive to the minimax rate of a suitably defined class of densities with heavy tails. To this aim, an approximation scheme different than the one used for location mixture is in order since the latter relies on the exponential tail condition \eqref{eq:8}. This remains an open problem and is left as an argument for future research.

We discuss next how minimax rates for density estimation depend on the tail.
It is well known that smoothness alone is not sufficient in order to guarantee consistency of density estimators in the $L_1$-norm.
In fact, there is a vast literature on minimax and adaptive minimax density estimation with $L_p$ norm which indicates the existence of a  tail zone, i.e. a range of values of $p$, for which the minimax rate depends on $p$ and deteriorates to $1$ as $p$ decreases to $1$. Such phenomenon does not appear in density estimation on a compact domain, see \citet{Gol:Lep:14} for an up to date literature review. 
\citet{Gol:Lep:14} have also determined a  tail dominance condition which illustrates how the tail zone shrinks to an empty set according to the tails of the density. The result which is relevant to our study is about the minmax rate under $L_1$-norm for $\beta$-Holder classes of densities which is, up to  $\log n$ factors, given by
\begin{equation}\label{eq:heavy}
  \max\left\{n^{-\frac{\beta}{2\beta+d}},\
  n^{-\frac{1-\theta}{1+d/\beta}}\right\}
\end{equation}   
where $\theta$ is a positive parameter in $(0,1]$ which determines the heaviness of the tails. See Remark 4.3 in \citet{Gol:Lep:14}.
The fastest rate $n^{-\beta/(2\beta+d)}$ is recovered for $\theta < \beta/(2\beta+d)$ (light tail), while for $\theta \geq \beta/(2\beta+d)$ (heavy tail) a wide range of slower minimax rates is obtained. See Theorem 13 of \citet{devroye:gyorfi:1985} for a closely related result in the light tail case. 
Clearly the exponential tail assumption \eqref{eq:8} corresponds to small value of $\theta$ and this explain why in \citet{shen:etal:2012} the usual $n^{-\beta/(2\beta+d)}$ rate is achieved.

\section{Discussion}
\label{sec:discussion}

In this paper we have discussed asymptotic properties of DP location-scale mixtures of Gaussian kernels for multivariate density estimation. 
To our knowledge, this is the first contribution to posterior asymptotics in the context of multivariate location-scale mixtures, a modelling approach which is a common practice in many applications.  
We have given sufficient conditions on the DP prior measure on the space of means and covariance matrices in order to achieve posterior consistency. 
While showing that these conditions are satisfied for widely used inverse-Wishart distribution, with the same practical motivation of providing theoretical justification for models used in practice, we showed that the conditions hold if one uses priors that parsimoniously model the covariance in high dimensional settings, such as a factor model or spectral decomposition  having both computational tractability and better fit in finite samples. 
Future work will deal with mixtures of more general random probability measures, like the two-parameter Poisson-Dirichlet process \citep{perman:etal:1992}, the normalized inverse Gaussian process \citep{lijoi:etal:2005} and, in general, Gibbs-type priors \citep{gnedin:pitman:2005}. A characteristic feature of Gibbs-type priors is given by their heavy-tailedness, which has certain advantages in terms of statistical inference. See \citet{deblasi:etal:2013}. From the perspective of frequentist asymptotics this implies that the sieve based on the truncated stick-breaking construction does not carry over in a straightforward way and hence this problem deserves further investigations.

\section*{Acknowledgements}
Comments on an earlier version of this work from David B. Dunson are greatly appreciated. 
The authors also thank the editor, the associate editor and the reviewers for comments that have helped to improve the paper substantially,
Alberto Zanella for providing useful references about the condition number distribution and Bertrand Lods, Bernardo Nipoti, and Igor Pr\"unster for fruitful discussions.
The present work is supported by Regione Piemonte and by the European Research Council (ERC) through StG \comillas{N-BNP} 306406.

\section*{Appendix}
\begin{proof}[Proof of Lemma \ref{lemma:KLsupport}]
It is well known that $P\in\mathscr{P}$ belongs to the weak support of $\Pi^{*}$ (meaning that any weak neighborhood of $P$ has positive probability under $\Pi^{*}$) if the support of $P$ is contained in the support of $P^*$. We recall that the support of $P$ is the smallest closed set of $P$-measure 1. Hence, since the support of $P^*$ is the whole space $\R^d\times\mathcal{S}$, the weak support of $\Pi^{*}$ is the whole space $\mathscr{P}$ and 
it contains, in particular, any compactly supported $P$.

The proofs of Theorem 5 and Theorem 2 of \citet{art:wu:ghos:2008} imply that, under the conditions of Lemma~\ref{lemma:KLsupport}, for any $\epsilon>0$, there exists $P_\epsilon\in\mathscr{P}$ such that $\int f_0\log(f_0/f_{P_\epsilon})\leq \epsilon$. In particular, $P_\epsilon$ can be taken as compactly supported, that is $P_\epsilon(D)=1$ where $D = [-a,a]^p \times \{\Sigma\in \mathcal{S}: \sigmadown^2\leq \lambda_i(\Sigma)\leq \sigmaup^2, i=1,\ldots,d \})$ for some constants $a>0$, $0 <\sigmadown < \sigmaup $.
%
%
Since
  $$\int f_0(x) \log \frac{f_0(x)}{f(x)} \ddr x = 
  \int f_0(x) \log \frac{f_0(x)}{f_{P_\epsilon}(x)} \ddr x + 
  \int f_0(x) \log \frac{f_{P_\epsilon}(x)}{f(x)} \ddr x$$
in order to prove \eqref{eq:KL}, we next show that there exists $W\subset\mathscr{P}$ with $\Pi^{*}(W)>0$ such that, for any $P \in W$,
  $\int f_0 \log (f_{P_\epsilon}/f_P) \leq\epsilon$.
To this aim, we verify the hypotheses of Lemma 3 in \citet{art:wu:ghos:2008}. 
It is clear that $P_\epsilon$ belongs to the weak support of $\Pi^{*}$. Next, condition (A7) of Lemma 3 in \citet{art:wu:ghos:2008} 
%
%
requires $\log f_{P_\epsilon}$  and $\log\inf_{(\mu,\Sigma)\in D}\phi_\Sigma(x-\mu)$ to be $f_0$-integrable. 
Note that, for $\|x\|<a$, $\inf_{(\mu,\Sigma)\in D}\phi_\Sigma(x-\mu)$ is bounded, while, for $\|x\|\geq a$,
  $$\inf_{(\mu,\Sigma)\in D}\phi_\Sigma(x-\mu)
  =\sigmaup^{-d} 
  \exp\left\{-\frac{4\|x\|^2}{2\sigmadown^2}\right\}.$$
Hence, under the hypotheses made, $\log\inf_{(\mu,\Sigma)\in D}\phi_\Sigma(x-\mu)$ is $f_0$-integrable. Also
  $$|\log f_{P_\epsilon}|
  \leq\left|
  \log\left\{
  \sigmaup^{-d} 
  \exp\left\{-\frac{4\|x\|^2}{2\sigmadown^2}\right\} P_\epsilon(D)
  \right\}
  \right|$$
for $\|x\|\geq a$, so that $\log f_{P_\epsilon}$ is $f_0$-integrable by a similar argument.

As for condition (A8) of Lemma 3 in \citet{art:wu:ghos:2008}, it is obviously  satisfied since the multivariate normal kernel $\phi_\Sigma(x-\mu)$ is bounded away from zero for $x$ in a compact set of $\R^d$ and $(\mu,\Sigma)\in D$. Finally, condition (A9) of Lemma 3 in \citet{art:wu:ghos:2008} requires that, for $C\subset\R^d$ a given compact set,  
  $\{\phi_\Sigma(x-\mu), x\in C\}$
is uniformly equicontinuous as a family of functions of $(\mu,\Sigma)$ on $D$. This can be shown by adapting to the present context the arguments given in the last part of the proof of Theorem 2 of \citet{art:wu:ghos:2008}.
\end{proof}
%


%
\begin{proof}[Proof of Theorem \ref{th:Gho07}]
According to Corollary 1 of \citet{ghos:vand:2007}, for any set of probability measures $\mathcal{Q}$ with 
  $\inf\{d(F_0,Q):\ Q\in\mathcal{Q}\}\geq 4\epsilon$,
any $\alpha,\beta>0$ and all $n\geq 1$, there exists a test $\phi_n$ such that
\begin{equation}\label{eq:eq:corr1}
  F_0^n\phi_n\leq\sqrt{\frac{\beta}{\alpha}}N(\epsilon,\mathcal{Q},d)
  \edr^{-n\epsilon^2},\quad
  \sup_{Q\in\mathcal{Q}}Q^n(1-\phi_n)\leq
  \sqrt{\frac{\alpha}{\beta}}\edr^{-n\epsilon^2}.
\end{equation}
Let $A_\epsilon=\{f:\ d(f_0,f)\leq8\epsilon\}$. Write
\begin{equation}\label{eq:break}
  \Pi(A_\epsilon^c|X_1,\ldots X_n)
  =\Pi(A_\epsilon^c\cap\F_n|X_1,\ldots X_n)
  +\Pi(A_\epsilon^c\cap\F_n^c|X_1,\ldots X_n).
\end{equation}  
From Lemma 4 of \citet{barr:etal:1999}, \eqref{eq:KL} implies that, for every $\eta>0$
\begin{equation}\label{eq:barron}
  F_0^\infty\left\{\Small{\int\prod_{i\leq n}}
  \left(f(x_i)/f_0(x_i)\right)\ddr\Pi(f)
  \leq \exp(-n\eta),\; i.o.\right\}=0,
\end{equation}  
so that, under \eqref{eq:complement}, 
the second expression on the r.h.s. of \eqref{eq:break} goes to $0$ in $F_0^n$-probability (actually in $F_0^\infty$-almost surely). So it is sufficient to prove that the first expression in the r.h.s. of \eqref{eq:break} goes  to $0$ in $F_0^n$-probability. Let $R_n(f)=\prod_{i\leq n}f(X_i)/f_0(X_i)$ and
  $I_n=\int R_n(f)\ddr\Pi(f)$. 
The event $B_n$ that $I_n\geq \edr^{-cn\epsilon^2}$ for $c$ in \eqref{eq:summability1} satisfies $F_0^n(B_n)\to 1$, see \eqref{eq:barron}. Therefore
  $$F_0^n\left[ \Pi(A_\epsilon^c\cap\F_n|X_1,\ldots,X_n)\right]
  =F_0^n\left[ \Pi(A_\epsilon^c\cap\F_n|X_1,\ldots,X_n)
  \Indic_{B_n}\right]+o_p(1).$$
For arbitrary tests $\phi_{n,j}$, we have 
\begin{align*}
  F_0^n\left[ \Pi(A_\epsilon^c\cap\F_{n,j}|X_1,\ldots,X_n)\Indic_{B_n}\right]
  &\leq F_0^n \phi_{n,j}+F_0^n\left(
  (1-\phi_{n,j})\Small{\int_{A_\epsilon^c\cap\F_{n,j}}}
  R_n(f)\ddr\Pi(f)\right)\edr^{cn\epsilon^2}\\
  &\leq F_0^n \phi_{n,j}
  +\sup_{f\in\F_{n,j}: d(f_0,f)\geq 8\epsilon}
  F^n(1-\phi_{n,j})\Pi(\F_{n,j})\edr^{cn\epsilon^2}
\end{align*} 
which can be bounded by a multiple of
  $$\sqrt{\frac{\beta_j}{\alpha_j}}
  N(2\epsilon,\F_{n,j},d)\edr^{-4n\epsilon^2}
  +\sqrt{\frac{\alpha_j}{\beta_j}}
  \edr^{-(4-c)n\epsilon^2}\Pi(\F_{n,j})$$
for the choice of $\phi_{n,j}$ from Corollary 1 with $2\epsilon$ in place of $\epsilon$, $\mathcal{Q}=\{F\in\F_{n,j}:d(f_0,f)\geq 8\epsilon\}$ and any $\alpha_j,\beta_j>0$. Put $\alpha_j=N(2\epsilon,\F_{n,j},d),\beta_j=\Pi(\F_{n,j})$ 
to obtain
\begin{multline*}
  F_0^n\left[ \Pi(A_\epsilon^c\cap\F_n|X_1,\ldots,X_n)\Indic_{B_n}\right]
  \leq \Small{\sum_j}
  \sqrt{N(2\epsilon,\F_{n,j},d)}\sqrt{\Pi(\F_{n,j})} 
  \edr^{-4n\epsilon^2}\\
  +\Small{\sum_j}
  \sqrt{N(2\epsilon,\F_{n,j},d)}\sqrt{\Pi(\F_{n,j})} 
  \edr^{-(4-c)n\epsilon^2}
\end{multline*}  
so that \eqref{eq:summability1} yields the result.
\end{proof}
%

%
%
\begin{lemma}\label{lemma:sieve}
Let $H,M\in\N$, $\sigmadown>0$ and, for $h=1,\ldots,H$, $0\leq\underline{a}_h<\bar{a}_h$, $1\leq u_h$. Define
\begin{multline*}
  \mathcal{G}=\{f_P\mbox{ with }P=\Small{\sum}_{h\geq 1}
  \pi_h\delta_{(\mu_h,\Sigma_h)}:\ 
  \Small{\sum}_{h>H}\pi_h\leq \epsilon;\
  \mbox{ for } h\leq H,\ 
  \underline{a}_h<\|\mu_h\|\leq \bar{a}_h,
  \\
  \sigmadown^2\leq \lambda_d(\Sigma_h),\lambda_1(\Sigma_h)
  \leq \sigmadown^2(1+\epsilon/\sqrt d)^M,\
  \lambda_1(\Sigma_h)/\lambda_d(\Sigma_h)\leq 
  u_h
  \}
\end{multline*}
Then, for some positive constant $C_1$,
\begin{equation}\label{eq:6}
  \begin{split}
  N(\epsilon,\mathcal{G},\|\cdot\|_1)\lesssim
  &\exp\Big\{\Small{
  dH\log M
  +H\log\frac{C_1}{\epsilon}}\\
  &+\Small{
  \sum_{h\leq H}
  \log\Big[
  \Big(\frac{\bar{a}_h}{\sigmadown\epsilon/2}+1\Big)^d
  -\Big(\frac{\underline{a}_h}{\sigmadown\epsilon/2}-1\Big)^d\Big]
  +\frac{d(d-1)}{2}\log\frac{2d\,u_h}{\epsilon^2}}
  \Big\}.
  \end{split}
\end{equation} 
\end{lemma}
\begin{proof}
The proof of \eqref{eq:6} is based on a modification of the arguments of Proposition 2 in \citet{shen:etal:2012} to the location-scale mixture case. 
We recall here that a set $\hat{\mathcal{G}}\subset\mathcal{G}$ with the property that any element of ${\cal G}$  is within $\epsilon$-distance from an element of $\hat{\cal G}$ is called an {\it $\epsilon$-net} over ${\cal G}$. Since $N(\epsilon,{\cal G},d)$ is the minimal cardinality of an $\epsilon$-net over ${\cal G}$, $N(\epsilon,{\cal G},d)\leq\#(\hat{\mathcal{G}})$. 

Let $P_1=\sum_{h\geq 1}\pi_h^{(1)}\delta_{(\mu_h^{(1)},\Sigma_h^{(1)})}$ and $P_2=\sum_{h\geq 1}\pi_h^{(2)}\delta_{(\mu_h^{(2)},\Sigma_h^{(2)})}$. 
Then
\begin{small}
\begin{align*}
  \|f_{P_1}-f_{P_2}\|_1
  \leq&
  \sum_{h\leq H}\pi_h^{(1)}\left\|\phi_{\Sigma_h^{(1)}}
  (\cdot-\mu_h^{(1)})-\phi_{\Sigma_h^{(2)}}
  (\cdot-\mu_h^{(2)})  
  \right\|_1
  +\sum_{h\leq H}\left|\pi_h^{(1)}-\pi_h^{(2)}\right|
  +\sum_{h> H}\left\{\pi_h^{(1)}+\pi_h^{(2)}\right\}.
\end{align*}
\end{small}
Note that
\begin{small}
\begin{align*}
  \left\|\phi_{\Sigma_h^{(1)}}
  (\cdot-\mu_h^{(1)})-\phi_{\Sigma_h^{(2)}}
  (\cdot-\mu_h^{(2)})  
  \right\|_1
  &\leq
  \left\|\phi_{\Sigma_h^{(2)}}
  (\cdot-\mu_h^{(1)})-\phi_{\Sigma_h^{(2)}}
  (\cdot-\mu_h^{(2)})  
  \right\|_1
  +
  \left\|\phi_{\Sigma_h^{(1)}}
  -\phi_{\Sigma_h^{(2)}}
  \right\|_1.
\end{align*}
\end{small}
The first term in the r.h.s. is smaller than
  $\sqrt{2/\pi}\|\mu_h^{(1)}-\mu_h^{(2)}\|
  \big/\sqrt{\lambda_d(\Sigma_h^{(2)})}$. 
As for the second term, use spectral decompositions $\Sigma_h^{(j)}=[O_h^{(j)}\Lambda_h^{(j)}(O_h^{(j)})^T]^{-1}$, $j=1,2$,  where $O_h^{(1)},O_h^{(2)}$ are orthogonal matrices and $\Lambda_h^{(1)},\Lambda_h^{(2)}$ are diagonal matrices. 
Drop the index $h$ so to ease the notation and write $\Sigma_h^{(j)}=\Sigma_j$, $O_h^{(j)}=O_j$ and $\Lambda_h^{(j)}=\Lambda_j$. Moreover, let $\lambda_{j,1}\geq\ldots\geq\lambda_{j,d}$ be the diagonal elements of $\Lambda_j$.
Finally, define $\widetilde{\Sigma}=(O_2\Lambda_1O_2^T)^{-1}$. By triangular inequality
\begin{equation}\label{eq:triang}
  \left\|\phi_{\Sigma_1}-\phi_{\Sigma_2}\right\|_1
  \leq 
  \left\|\phi_{\widetilde\Sigma}-\phi_{\Sigma_2}\right\|_1
  +\left\|\phi_{\widetilde\Sigma}-\phi_{\Sigma_1}\right\|_1.
\end{equation}  
For each term in the r.h.s., use Csisz\'ar's inequality,
  $\|f-g\|^2_1\leq 2 \Small{\int}\log(f/g)f$,
and the exact expression of the Kullback-Leibler divergence between zero mean multivariate Gaussians to get
  $$\Small{\int}\log(\phi_{\Sigma_2}/\phi_{\Sigma_1})\phi_{\Sigma_2}
  ={1\over 2}
  \left(\tr(\Sigma_1^{-1}\Sigma_2)-\log\det(\Sigma_1^{-1}\Sigma_2)-d
  \right).$$
As for the first term in the r.h.s. of \eqref{eq:triang},
\begin{align*}
  \left\|\phi_{\widetilde\Sigma}-\phi_{\Sigma_2}\right\|_1
  &\leq
  \left\{\tr(\Sigma_2^{-1}\widetilde\Sigma)-\log\det(\Sigma_2^{-1}
  \widetilde\Sigma)-d\right\}^{1/2}\\
  &=  \left\{\Small{\sum}_{i=1}^d\left(
  \lambda_{2,i}/\lambda_{1,i}
  -\log\Small{\lambda_{2,i}/\lambda_{1,i}}-1\right)\right\}^{1/2}
\end{align*}
where $\lambda_{2,i}/\lambda_{1,i}$ corresponds to $\lambda_{d-i+1}(\Sigma_h^{(1)})/\lambda_{d-i+1}(\Sigma_h^{(2)})$ in the original notation. As for the second term in the r.h.s. of \eqref{eq:triang},
\begin{align*}
  \left\|\phi_{\widetilde\Sigma}-\phi_{\Sigma_1}\right\|_1
  &\leq
  \left\{\tr(\Sigma_1^{-1}\widetilde\Sigma)-\log\det(\Sigma_1^{-1}
  \widetilde\Sigma)-d\right\}^{1/2}\\
  &=\left\{\tr(Q\Lambda_1Q^T\Lambda_1^{-1})
  -d\right\}^{1/2}
  =\left\{\tr(Q\Lambda_1Q^T\Lambda_1^{-1}-I)\right\}^{1/2}
\end{align*}
where $Q=O_2^TO_1$. Let $B=Q-I$ and $\|B\|_{\max}=\max|b_{i,j}|$. 
Then 
  $$\|B\|_{\max}\leq \|B\|_2=\|O_2^TO_1-I\|_2
  \leq\|O_2^T\|_2\|O_1-O_2\|_2=\|O_1-O_2\|_2.$$
Let $b_{ij}$ be the element $(i,j)$ of $B$.
Note that
\begin{align*}
  \tr(Q\Lambda_1Q^T\Lambda_1^{-1}-I)
  &=\tr(B+\Lambda_1B^T\Lambda_1^{-1}+B\Lambda_1B^T\Lambda_1^{-1})\\
  &=\tr(B)+\tr(\Lambda_1B^T\Lambda_1^{-1})+
  \tr(B\Lambda_1B^T\Lambda_1^{-1})\\
  &=\tr(B)+\tr(B^T)+
  \tr(B\Lambda_1B^T\Lambda_1^{-1})\\
  &\leq 
  2\tr(B)+\Small{\frac{\lambda_{11}}{\lambda_{1d}}} \tr(BB^T)\\
  &=2\tr(B)+2 \Small{\frac{\lambda_{11}}{\lambda_{1d}}} \tr(I-Q)
  =2\tr(B)\left[1-\Small{\frac{\lambda_{11}}{\lambda_{1d}}}\right]\\
  &\leq 2d\|B\|_{\max} \left[
  \Small{\frac{\lambda_{1,1}}{\lambda_{1,d}}}-1\right]
  \leq 2d\|B\|_{\max}\Small{\frac{\lambda_{1,1}}{\lambda_{1,d}}}
\end{align*}  
where in the second last inequality we use $\tr(B)\geq -d\|B\|_{\max}$ together with $\lambda_{1,1}/\lambda_{1,d}\geq 1$. 
Hence,
  $$\left\|\phi_{\widetilde\Sigma}-\phi_{\Sigma_1}\right\|_1
  \leq\left\{2d\|O_1-O_2\|_2 
  \Small{\frac{\lambda_{1,1}}{\lambda_{1,d}}}
  \right\}^{1/2}$$
where $\lambda_{1,1}/\lambda_{1,d}$ corresponds to $\lambda_1(\Sigma_h^{(1)})/\lambda_d(\Sigma_h^{(1)})$ in the original notation.
In summary, back to the original notation, 
\begin{multline*}
  \|f_{P_1}-f_{P_2}\|_1
  \leq \sum_{h\leq H}\pi_h^{(1)} 
  \left[
  \sqrt{\frac{2}{\pi}}\frac{\|\mu_h^{(1)}-\mu_h^{(2)}\|}
  {\sqrt{\lambda_d(\Sigma_h^{(2)})}}+
  \left\{\sum_{i=1}^d\left(\frac{\lambda_i(\Sigma_h^{(1)})}
  {\lambda_i(\Sigma_h^{(2)})}
  -\log\frac{\lambda_i(\Sigma_h^{(1)})}
  {\lambda_i(\Sigma_h^{(2)})}-1\right)\right\}^{1/2}\right.\\
  \left.+\left\{2d\|O_h^{(1)}-O_h^{(2)}\|_2 
  \frac{\lambda_1(\Sigma_h^{(1)})}{\lambda_d(\Sigma_h^{(1)})}
  \right\}^{1/2}
  \right]
  +\sum_{h\leq H}\left|\pi_h^{(1)}-\pi_h^{(2)}\right|
  +\sum_{h> H}\left\{\pi_h^{(1)}+\pi_h^{(2)}\right\}.
\end{multline*}
Now pick $f_P\in{\cal G}$ with $P=\sum_{h\geq 1}\pi_h\delta_{(\mu_h,\Sigma_h)}$, $\Sigma_h=(O_h\Lambda_hO_h^T)^{-1}$ and $\Lambda_h=\mbox{diag}(\lambda_{h,1},\ldots,\lambda_{h,d})$. Then find\\
$\bullet$ 
$\hat\mu_h\in\hat{R}_h$, $h=1,\ldots,H$, where $\hat{R}_h$ is a $\sigmadown\epsilon$-net of $R_h=\{\mu\in \R^d:\ \underline{a}_h<\|\mu\|\leq \bar{a}_h\}$ such that
  $\|\mu_h-\hat \mu_h\|<\sigmadown\epsilon;$\\
$\bullet$ $\hat\pi_1,\ldots,\hat\pi_H\in\hat \Delta$, where $\hat \Delta$ is a $\epsilon$-net of the $H$-dimensional probability simplex $\Delta$ 
such that
  $\Small{\sum}_{h\leq H}\left|\tilde\pi_h-\hat\pi_h \right|\leq\epsilon,$
and $\tilde\pi_h=\pi_h\Big/\sum_{l\leq H}\pi_l$, $h\leq H$;\\  
$\bullet$ $\hat O_h\in\hat{\mathcal{O}_h}$, $h=1,\ldots,H$, where $\hat{\mathcal{O}_h}$ is a $\delta_h$-net of the set $\mathcal{O}$ of $d\times d$ orthogonal matrices with respect to the spectral norm $\|\cdot\|_2$, with
  $\delta_h=\epsilon^2/(2d\,u_h)$
such that 
  $\|O_h-\hat O_h\|_2\leq \delta_h$;\\
$\bullet$ $(m_{h,1},\ldots,m_{h,d})\in\{1,\ldots,M\}^d$, $h=1,\ldots,H$, such that 
  $\hat\lambda_{h,i}=\{\sigmadown^2(1+\epsilon/\sqrt d)^{m_{h,i}-1}\}^{-1}$
satisfies
  $1\leq \hat\lambda_{h,i}/\lambda_{h,i}<(1+\epsilon/\sqrt d).$\\  
Take $\hat P=\sum_{h\leq H}\hat\pi_h\delta_{(\hat \mu_h,\hat\Sigma_h)}$ and $\hat\Sigma_h=(\hat O_h\hat\Lambda_h\hat O_h^T)^{-1}$, where $\hat\Lambda_h=\diag(\hat\lambda_{h,1},\ldots,\hat\lambda_{h,d})$. Then
\begin{align*}
  \|f_{P}-f_{\hat P}\|_1
  \leq& 
  \max_{h\leq H}
  \left[
  \sqrt{\frac{2}{\pi}}\frac{\|\mu_h-\hat \mu_h\|}
  {\sqrt{\lambda_d(\hat\Sigma_h)}}+
  \left\{\sum_{i=1}^d\left(\frac{\lambda_i(\Sigma_h)}
  {\lambda_i(\hat\Sigma_h)}
  -\log\frac{\lambda_i(\Sigma_h)}
  {\lambda_i(\hat\Sigma_h)}-1\right)\right\}^{1/2}\right]\\
  &
  +\sum_{h\leq H}\pi_h
  \left\{2d\|O_h-\hat O_h\|_2 
  \frac{\lambda_1(\Sigma_h)}{\lambda_d(\hat\Sigma_h)}
  \right\}^{1/2}
  +\sum_{h\leq H}\left|\pi_h-\hat\pi_h\right|
  +\sum_{h> H}\pi_h\\
  \leq& 
  \max_{h\leq H}
  \left[
  \sqrt{\frac{2}{\pi}}\frac{\|\mu_h-\hat \mu_h\|}
  {\sigmadown}+
  \left\{\sum_{i=1}^d\left(\frac{\hat\lambda_{h,d-i+1}}
  {\lambda_{h,d-i+1}}
  -1\right)^2\right\}^{1/2}\right]\\
  &
  +
  \sum_{h\leq H}\pi_h
  \left\{2d\|O_h-\hat O_h\|_2 
  \frac{\lambda_{h,1}}{\lambda_{h,d}}
  \right\}^{1/2}
  +\sum_{h\leq H}\left|\pi_h-\hat\pi_h\right|
  +\sum_{h> H}\pi_h,  
\end{align*}
where in the second inequality $x-\log x-1\leq (x-1)^2$ for $x\geq1$ together with $\hat\lambda_{h,d-i+1}/\lambda_{h,d-i+1}\geq 1$ has been used to bound the second term. Moreover, 
\begin{align*}
  \sum_{h\leq H}\left|\pi_h-\hat\pi_h\right|
  &\leq\sum_{h\leq H}\left|\pi_h-(1-\Small{\sum}_{h>H}\pi_h)\hat\pi_h\right|+
  \sum_{h\leq H}\left|(1-\Small{\sum}_{h>H}\pi_h)\hat\pi_h-\hat\pi_h\right|\\
  &=(1-\Small{\sum}_{h>H}\pi_h)
  \sum_{h\leq H}\left|\tilde\pi_h-\hat\pi_h\right|
  +(\Small{\sum}_{h>H}\pi_h)\sum_{h\leq H}\hat\pi_h\\
  &\leq \sum_{h\leq H}\left|\tilde\pi_h-\hat\pi_h\right|
  +\sum_{h>H}\pi_h\leq 2\epsilon.
\end{align*}  
Hence, 
\begin{align*}
  \|f_{P}-f_{\hat P}\|_1
  &\leq 
  \sqrt{\frac{2}{\pi}}\epsilon+
  \left\{\sum_{i=1}^d\left(1+\frac{\epsilon}{\sqrt d}-1\right)^2
  \right\}^{1/2}
  +\sum_{h\leq H}\pi_h
  \left\{2d\frac{\epsilon^2}
  {2d\,u_h}\frac{\lambda_{h,1}}{\lambda_{h,d}} \right\}^{1/2}
  +2\epsilon
  +\epsilon\\
  &\leq \epsilon+\{\epsilon^2\}^{1/2}+\{\epsilon^2\}^{1/2}
  +3\epsilon=6\epsilon.
\end{align*}
Thus a $6\epsilon$-net of ${\cal G}$, in the $L_1$ distance, can be constructed with $f_{\hat P}$ as above. Recalling that $\#(\hat\Delta)\lesssim \epsilon^{-H}$, $\#(\hat{\mathcal{O}_h})\lesssim \delta_h^{-\frac{d(d-1)}{2}}$ and
  $\#(\hat R_h)\leq 
  (\bar{a}_h/(\sigmadown\epsilon/2)+1)^d
  -(\underline{a}_h/(\sigmadown\epsilon/2)-1)^d$, 
the total number is
  $$\lesssim
  (M)^{dH}
  \epsilon^{-H}
  \prod_{h\leq H}
  \left[\left(\frac{\bar{a}_h}{\sigmadown\epsilon/2}+1\right)^d
  -\left(\frac{\underline{a}_h}{\sigmadown\epsilon/2}-1\right)^d\right]
  C_1\left(\frac{2d\,u_h}{\epsilon^2}\right)^{\frac{d(d-1)}{2}}
  $$
for some positive constant $C_1$. Finally, the constant factor by $6$ can be absorbed in the bound, hence \eqref{eq:6} follows.
\end{proof}
%
%
\begin{remark}
\label{rem:nogoshetal}
The usual low-entropy, high mass sieve approach of Theorem 2 by \citet{ghos:etal:1999} is not suitable for multivariate DP location-scale mixture. This theorem, indeed, requires a sieve $\F_n$ such that: (i) $\Pi(\F_n^c)$ decreases exponentially fast in $n$, and (ii) the $L_1$ metric entropy of $\F_n$ grows linearly with $n$. As for (i), let $\F_n$ to be defined as $\mathcal{G}$ of Lemma \ref{lemma:sieve} with $M=\sigmadown^{-2c_2}=\bar{a}_h=u_h=n$, $H_n=\lfloor Cn\epsilon^2/\log n\rfloor$, $\underline{a}_h$=0 and $l_h=1$. Clearly $\F_n \uparrow \F$ as $n\to\infty$.
However, in order to have $\F_n$ satisfying condition (i), stronger tail conditions on $P^*$ are necessary than those in \eqref{eq:1} and \eqref{eq:4}, namely exponential tail behaviors.
While one can be still satisfied with a Gaussian specification for $\mu$, exponential tail for the condition number turns out to be too restrictive ruling out all the models discussed in Section 3 as well as any other reasonable prior specification for non diagonal covariance matrices. However, as an inspection of the proof of Lemma \ref{lemma:sieve} reveals, with diagonal covariance matrices the metric entropy is considerably smaller and one can prove consistency by applying Theorem 2 of \citet{ghos:etal:1999} under essentially the same tail requirements of DP location mixture of \citet{shen:etal:2012}.  We omit the details here.	
\end{remark}
\begin{proof}[Proof of Corollary \ref{cor:consistencywhishart}]
That conditions  \eqref{eq:2} and \eqref{eq:3} are satisfied when $\mathcal{L}=IW(\Sigma_0,\nu)$ has been proved in Lemma 1 of \citet{shen:etal:2012}.
Therefore we focus here on condition \eqref{eq:4} on the condition number. Let for the moment $\Sigma_0=I$. We directly work with the joint distribution of the ordered eigenvalues of a Wishart-distributed matrix: 
  $$f_{\lambda_1,\dots, \lambda_d}(x_1,\dots, x_d) 
  = K \exp\left\{- \frac{1}{2}\sum_{i=1}^d x_i\right\} 
  \prod_{j=1}^d (x_j)^{(\nu-d-1)/2} 
  \prod_{i<j}(x_i - x_j),$$
over the set $\{(x_1,\ldots,x_d)\in(0,\infty)^d:\; x_1\geq\ldots\geq x_d\}$, where $K$ is a known normalizing constant. See Corollary 3.2.19, page 107, of \citet{muirhead:1982}. 
To show condition \eqref{eq:4}, let $z = \lambda_1/\lambda_d$
and apply the Jacobi's transformation formula for
  $g(\lambda_1,\dots, \lambda_d) = 
  (z, \lambda_2, \dots, \lambda_d).$
For $J_{g^{-1}}$ denoting the Jacobian of $g^{-1}$, we have
  $$f_{z,\lambda_2,\dots,\lambda_d} (z,\dots, y)
  = |J_{g^{-1}}| f_{\lambda_1, \dots, \lambda_d}(x_dz,x_2,\dots, x_d),$$
and the marginal of $z$ can be obtained integrating out $x_2, \dots, x_d$. 
Using $|J_{g^{-1}}| = x_d$, $x_d\geq 0$ and $x_d\leq x_i\leq zx_d$ for $i=2,\ldots,d-1$, after some algebra one gets
  $$f_{z}(z)\leq K'' z^{-\frac{\nu-d+3}{2}}$$ 
for some constant $K''$, so that $P^*(z>x) \lesssim x^{-(\nu-d+1)/2}$ as $x\to\infty$. Cfr. Theorem 3.2 of \citet{edelman:sutton:2005}. If $\Sigma_0\neq I$, the conclusion holds for a different set of constants, see Theorem 4 of \citet{matthaiou:etal:2010} for the complex Wishart case. Condition \eqref{eq:4} is then satisfied under the hypothesis on $\nu$.
\end{proof}
\begin{proof}[Proof of Corollary \ref{cor:consistencyfactor}]
To prove condition \eqref{eq:2}, consider
  \[
  P^* \left\{ \lambda_1(\Sigma^{-1}) \geq x \right\} 
  \leq P^* \left\{ \text{tr}(\Sigma^{-1}) \geq x \right\}
  \leq P^* \left\{ \text{tr}(\Omega^{-1}) \geq x \right\},
  \]
where the last inequality  is verified since  $\text{tr}(\Sigma^{-1}) \leq \text{tr}(\Omega^{-1})$ by an application of the Woodbury's identity. Since the trace of $\Omega^{-1}$ is the sum of i.i.d. gamma random variable, and the gamma has exponential tail, \eqref{eq:2} is verified.
As for \eqref{eq:3}, consider
  \[
  P^* \left\{ \lambda_d(\Sigma^{-1}) < 1/x \right\} 
  \leq 	x^{-1} E_{P^*} \left\{ \lambda_1(\Sigma) \right\} 
  \leq x^{-1} \left[E_{P^*} \left\{ \lambda_1(\Omega)\right\}  
  + E_{P^*} \left\{\lambda_1(\Gamma \Gamma^T) \right\}\right]
  \]
by Markov's and Weyl's inequalities. The expectation of $\lambda_1(\Omega)$ is finite and since 
$\lambda_1(\Gamma \Gamma^T) \leq \mbox{tr}(\Gamma \Gamma^T) \sim \chi^2_{rd}$, also the expectation of $\lambda_1(\Gamma \Gamma^T)$ is finite.
Finally, to prove condition \eqref{eq:4}, first use  Markov's inequality. For $k > d(d-1)$, we have
\begin{align*}
  P^* \left\{ \lambda_1(\Sigma^{-1})/\lambda_d(\Sigma^{-1}) 
  \geq x \right\} 
  & = P^* \left\{
  \left(\lambda_1(\Sigma)/\lambda_d(\Sigma)\right)^k 
  \geq x^k \right\} \\ & \leq
  x^{-k} E \left\{ \left(\lambda_1(\Sigma)
  /\lambda_d(\Sigma)\right) ^k  \right\} . 
\end{align*}
Using again Weyl's inequalities and noting that $\lambda_d(\Gamma\Gamma^T)=0$ when $r<d$, we have 
  \[
  \frac{\lambda_1(\Sigma)}{\lambda_d(\Sigma)} \leq 
  \frac{\lambda_1(\Omega)+\lambda_1(\Gamma \Gamma^T)}
  {\lambda_d(\Omega)+\lambda_d(\Gamma \Gamma^T)} 
  =
  \frac{\lambda_1(\Omega)+\lambda_1(\Gamma \Gamma^T)}{\lambda_d(\Omega)} 
  \leq
  \frac{\lambda_1(\Omega^{-1})}{\lambda_d(\Omega^{-1})} 
  + \frac{\tr(\Gamma \Gamma^T)}{\lambda_d(\Omega)}.
  \]
By convexity of $g(x) = x^k$, for any $x,y \in \R^+$ we have $ (x+y)^k \leq 2^{k-1}(x^k + y^k)$. Hence $\lambda_1(\Sigma)/\lambda_d(\Sigma)$ has finite $k$-th moment if both summand have finite $k$-th moment. 
First consider the distribution of $\lambda_1(\Omega^{-1})/\lambda_d(\Omega^{-1})$.
Since $\Omega^{-1}$ is a diagonal matrix with iid gamma distributed diagonal elements, the joint distribution of the ordered eigenvalues of $\Omega^{-1}$ is the joint distribution of the ordered statistics of an iid sample from $\mbox{Ga}(a,b)$, i.e.
  $$	
  f_{\lambda_1(\Omega^{-1}) \dots \lambda_d(\Omega^{-1})}
  (x_1, \dots, x_d) = 
  d! f_{\sigma^{-2}_1}(x_1)\times \ldots\times f_{\sigma^{-2}_d}(x_d) = 
  K \prod_{j=1}^d x_j^{a-1} \exp\left\{-b \sum_{j=1}^d x_j\right\},
  $$
over the set $\{(x_1,\ldots,x_d)\in(0,\infty)^d:\; x_1\geq\ldots\geq x_d\}$, where $K$ is a known normalizing constant. 
Now we use the Jacobi's transformation formula as in the proof of Corollary~\ref{cor:consistencywhishart} to get, for $z=\lambda_1(\Omega^{-1})/\lambda_d(\Omega^{-1})$,
\begin{align*}
  f_z(z)  = 
  & K  \int_{x_3}^\infty \dots 
  \int_{x_{d-2}}^\infty \int_{x_d}^\infty 
  \prod_{j=2}^{d-1} x_j^{a-1} 
  \exp\left\{-b \sum_{j=1}^d x_j\right\} 
  \ddr x_2 \dots \ddr x_{d-1} \times \\
  & \int_0^\infty x_d (z x_d^2)^{a-1} \exp\{-bx_d(z+1)\} \ddr x_d 
  \leq \ldots\leq K' \frac{z^{a-1}}{(z+1)^{2a}}.
\end{align*}
Hence $\lambda_1(\Omega^{-1})/\lambda_d(\Omega^{-1})$ has finite $k$-th moment as long as $a > k$.  
Now, since
  \[
  \frac{\text{tr}(\Gamma \Gamma^T)}{\lambda_d(\Omega)} 
  = \text{tr}(\Gamma \Gamma^T) \lambda_1(\Omega^{-1})
  \]
and the $k$-th moment of the product of independent random variables is the product of the $k$-th moments, it remains to show that both $\text{tr}(\Gamma \Gamma^T)$ and $\lambda_1(\Omega^{-1})$ have finite $k$-th moment. 
To this aim, it is sufficient to note that $\text{tr}(\Gamma \Gamma^T) \sim \chi^2_{dr}$  and $\lambda_1(\Omega^{-1})$ has the distribution of the first order statistics of a sample of $d$ independent gamma random variables. The proof is then complete.
\end{proof}
\begin{proof}[Proof of Corollary \ref{remarkandrew}]
Since each eigenvalues has independent inverse-gamma distribution, conditions \eqref{eq:2}-\eqref{eq:3}-\eqref{eq:4} are satisfied following part of the proof of Corollary \ref{cor:consistencyfactor}.
\end{proof}
\begin{proof}[Proof of Proposition \ref{prop:1}]
Let $M_n=\underline\sigma_n^{-2c_2}=n$  with $\epsilon_n=n^{-\gamma}(\log n)^{t}$ for $t_0=t+s$ and $H_n=\lfloor Cn\epsilon_n^2/\log n\rfloor$ for a positive constant to be determined later. Also define $\F_n$ and $\F_{n,\bj,\bl}$ as in the proof of Theorem 2 with $\epsilon=\epsilon_n$. 
We prove next that
\begin{align}
  \label{eq:pier1}
  &\Pi(\F_n^c)\lesssim \edr^{-4n\tilde\epsilon_n^2}\\
  \label{eq:pier2}
  &\sum\nolimits_{\bj,\bl}
  \sqrt{N(\epsilon_n, \F_{n,\bj,\bl},d)}
  \sqrt{\Pi(\F_{n,\bj,\bl})}\edr^{-n\epsilon_n^2}\to 0.
\end{align}
The thesis then follows by an application of Theorem 5 by \citet{ghos:vand:2007}, cfr. version in Theorem 3 of \citet{krui:2010}.

As for \eqref{eq:pier1}, refer to equation 
\eqref{eq:prior_complement}. 
Note that $n\epsilon_n^2=n^{1-2\gamma}(\log n)^{2t_0}$ and $H_n=C n^{1-2\gamma}(\log n)^{2t_0-1}$. We have
\begin{small}
\begin{align*}
  \left\{
  \frac{\edr \alpha}{H_n}\log\frac{1}{\epsilon_n}
  \right\}^{H_n}
  &
  \lesssim\exp\{-H_n\log H_n\}
  =\exp\left\{
  -C n^{1-2\gamma}(\log n)^{2t_0-1}
  \log (C n^{1-2\gamma}(\log n)^{2t_0-1})
  \right\}\\
  &\lesssim\exp\left\{
  -C n^{1-2\gamma}(\log n)^{2t_0-1}
  (1-2\gamma)\log n
  \right\}
  =\exp\left\{-(1-2\gamma)C n\epsilon_n^2\right\}\\
  H_n\edr^{-c_1\underline\sigma_n^{-2 c_2}}
  &=Cn^{1-2\gamma}(\log n)^{2t_0-1}\edr^{-nc_1}
  =o\left(
  \exp\left\{-n\epsilon_n^2\right\}
  \right)\\
  H_n\underline\sigma_n^{-2c_3}\left(1+\epsilon_n/\sqrt d \right)^{-c_3M_n}
  &\leq Cn^{c_3/c_2+1-2\gamma}(\log n)^{2t_0-1}
  \exp\{-c_3n\epsilon_n/(2\sqrt d)\}
  = o\left(
  \exp\left\{-n\epsilon_n^2\right\}\right)
\end{align*}
\end{small}
Therefore
  $\Pi(\F_n^c)
  \lesssim
  \exp\left\{-(1-2\gamma)C n\epsilon_n^2\right\}
  \lesssim \edr^{-4n\tilde\epsilon_n^2}$
for all large $n$ since $t<t_0$, 
Note that the rate at which the prior probability of the complement of the sieve vanishes is determined by the truncation level of the stick-breaking weights.

As for \eqref{eq:pier2}, refer to equation 
\eqref{eq:entr_shell}. 
Based on the inequality $d(f,g)^2\leq \| f-g \|_1$,
  $$N(\epsilon_n,\F_{n,\bj,\bl},d)
  \lesssim \exp\left\{
  dH_n\,\log n
  +H_n\log\Small{\frac{C_1}{\epsilon_n^2}}
  +\Small{\sum_{h\leq H_n}}
  \log\Small{\frac{n^{c_4d}j_h^{d-1}}{\epsilon_n^{2d}}}
  +\Small{\frac{d(d-1)}{2}}
  \log\Small{\frac{2d\,n^{2^{l_h}}}{\epsilon_n^4}}\right\}$$
Since
  $H_n\log(1/\epsilon_n)
  =\gamma C n\epsilon_n^2
  +o\big(n\epsilon_n^2\big)$, 
we also have
\begin{align*}
  N(\epsilon_n,\F_{n,\bj,\bl},d)
  &\lesssim \exp\left\{
  \left[d+c_4d+2\gamma+2\gamma d+2\gamma d(d-1)\right]
  C n\epsilon_n^2\right\}
  \Small{\prod_{h\leq H_n}}
  j_h^{d-1}n^{\frac{d(d-1)}{2}2^{l_h}}\\
  &< 
  \exp\left\{
  d^*C n\epsilon_n^2\right\}
  \Small{\prod_{h\leq H_n}}
  j_h^{d-1}n^{\frac{d(d-1)}{2}2^{l_h}}
\end{align*}
where $d^*=d+c_4d+1+d+d(d-1)$ and $\gamma< 1/2$ has been used in the last inequality. By using
\eqref{eq:prior_shell}, 
  $\sqrt{N(\epsilon_n, \F_{n,\bj,\bl},d)}
  \sqrt{\Pi(\F_{n,\bj,\bl})}$ 
can be bounded by a multiple of
  $$\exp\left\{ (d^*/2)C n\epsilon_n^2\right\}
  \Small{\prod_{h\leq H_n}}
  j_h^{\frac{d-1}{2}}[\sqrt{n}(j_h-1)]^{-\Indic_{(j_h\geq 2)}(r+1)}
  n^{\frac{d(d-1)}{4}2^{l_h}
  -\Indic_{(l_h\geq 1)}\frac{\kappa}{2}2^{l_h-1}}$$
Summing with respect to $\bj$ and $\bl$, by arguments similar to those used in the proof of Theorem \ref{th:consistency}, we obtain, for $n$ large enough and when $r>(d-1)/2$ and $\kappa>d(d-1)$, the upper bound
\begin{align*}
  \sum\nolimits_{\bj,\bl}\sqrt{N(\epsilon_n, \F_{n,\bj,\bl},d)}
  \sqrt{\Pi(\F_{n,\bj,\bl})}
  &\leq  \exp\left\{(d^*/2)
  C n\epsilon_n^2\right\}
  (1+n^{-\frac{r+1}{2}}K)^{H_n}
  \left[2n^{\frac{d(d-1)}{4}}\right]^{H_n}\\
  &\lesssim\exp\left\{\left(d^*/2+d(d-1)/4 \right)
  C n\epsilon_n^2\right\}
\end{align*}
where $K:=\sum_{j\geq2}j^{(d-1)/2}(j-1)^{-(r+1)}<\infty$.
Hence \eqref{eq:pier2} is satisfied 
for $C$ sufficiently small to satisfy 
  $C\leq 4/[2d^*+d(d-1)]$.
The proof is complete.  
\end{proof}
\begin{proof}[Proof of Proposition \ref{prop:2}]
The proof follows the arguments presented in Theorem 4 of \citet{shen:etal:2012} adapted to the location-scale case. 
Denote by $F_0\, g=\int g(x)f_0(x)\ddr x$ the expectation of $g(X)$ under $X\sim f_0$. Consider $n$ large enough such that $\tilde\epsilon_n<s_0^\beta$ and fix $\sigma_n^\beta=\tilde\epsilon_n \{\log(1/\tilde\epsilon_n)\}^{-1}$. As in Proposition 1 of \citet{shen:etal:2012}, define $E_{\sigma_n}=\{\|x\|\in\R^d:\ f_0(x)\geq \sigma_n^{(4\beta+2\epsilon+8)/\delta}\}$ for some $\delta\in(0,1)$ and $a_{\sigma_n}=a_0\{\log(1/\sigma_n)\}^{1/\tau}$. Then, $F_0(E_{\sigma_n}^c)\leq B_0\sigma_n^{4\beta+2\epsilon+8}$ for some constant $B_0$, $E_{\sigma_n}\subset\{\|x\|\in\R^d:\ \|x\|\leq a_{\sigma_n}\}$ and, there exists a density $\tilde h_{\sigma_n}$ with support $E_{\sigma_n}$ such that $d(f_0,\phi_{\sigma_n}\star \tilde h_{\sigma_n})\lesssim \sigma_n^{\beta}$.  
Find $b_1>\max\{1,1/(2\beta)\}$ such that 
  $\tilde\epsilon_n^{b_1}\{\log(1/\epsilon_n)\}^{5/4}
  \leq \tilde\epsilon_n$.
Apply Theorem B1 of \citet{shen:etal:2012} 
for $\sigma=\sigma_n$, $a=a_{\sigma_n}$, $F=\tilde h_{\sigma_n}$ and $\epsilon=\tilde\epsilon_n^{2b_1}$ to get 
  $F_{\sigma_n}=\sum_{j=1}^N p_j\delta_{\mu_j}$ with 
  $$N\leq D[(a_{\sigma_n}\sigma_n^{-1}\vee1)\log(1/\tilde\epsilon_n^{2b_1})]^d
  \lesssim \sigma_n^{-d}\{\log(1/\tilde\epsilon_n)\}^{d+d/\tau}$$
many support points inside $\{x\in\R:\ \|x\|\leq a_{\sigma_n}\}$ such that 
  $$\| \phi_{\sigma_n}\star\tilde h_{\sigma_n}-
  \phi_{\sigma_n}\star F_{\sigma_n} \|_1
  \lesssim \tilde\epsilon_n^{2b_1}
  \{\log(1/\tilde\epsilon_n^{2b_1})\}^{1/2}.$$
Move each of the support point of $F_{\sigma_n}$ to the nearest point on the grid 
  $\{(n_1,\ldots,n_d)\sigma_n\tilde\epsilon_n^{2b_1}: n_i\in\mathds{Z},
  |n_i|<\lceil a_{\sigma_n}/(\sigma_n\tilde\epsilon_n^{2b_1}) \rceil,
  i=1,\ldots,d\}$,
so that $\min_{j\neq l}\|\mu_j-\mu_l\|>\sigma_n\tilde\epsilon_n^{2b_1}$. These moves cost at most a constant times $\tilde\epsilon_n^{2b_1}$ to the $L_1$ distance. 
Modify further $F_{\sigma_n}$ so that $(\mu_1,\ldots,\mu_N)$ form a $\sigma_na_{\sigma_n}^{\tau/2}$-net of $\{x\in\R^d: \|x\|\leq a_{\sigma_n} \}$.
This can be achieved by adding at most 
  $[a_{\sigma_n}/(\sigma_n a_{\sigma_n}^{\tau/2})+1]^d
  \lesssim\sigma_n^{-d}a_{\sigma_n}^{d-d\tau/2}
  \lesssim\sigma_n^{-d}\{\log(1/\sigma_n)\}^{d/\tau-d/2}$ 
support points so that we still have
  $N\lesssim \sigma_n^{-d}\{\log(1/\tilde\epsilon_n)\}^{d+d/\tau}$.
Finally, modify the probability masses $p_1,\ldots,p_N$ to a version for which $p_j\geq \tilde\epsilon_n^{4db_1}$ for $j=1,\ldots,N$. 
By a suitable extension of Lemma 3 of \citet{krui:2010} to $d$ dimensions, the $L_1$ distance changes by the $L_1$ distance between the original and the new probability weights, and the latter can be shown to be smaller than $2(N-1)\tilde\epsilon_n^{4db_1}$. In conclusion, we have
\begin{align}\label{eq:discr}
  \| \phi_{\sigma_n}\star\tilde h_{\sigma_n}-
  \phi_{\sigma_n}\star F_{\sigma_n} \|_1
  &\lesssim \tilde\epsilon_n^{2b_1}\{\log(1/\tilde\epsilon_n^{2b_1})\}^{1/2}
  +\tilde\epsilon_n^{2b_1}+N \tilde\epsilon_n^{4db_1}\nonumber \\
  &\lesssim \tilde\epsilon_n^{2b_1}\{\log(1/\tilde\epsilon_n)\}^{1/2}
  +\sigma_n^{-d}\{\log(1/\tilde\epsilon_n)\}^{d+d/\tau}
  \tilde\epsilon_n^{4db_1}\nonumber \\
  &= \tilde\epsilon_n^{2b_1}\{\log(1/\tilde\epsilon_n)\}^{1/2}
  +\tilde\epsilon_n^{d(4b_1-1/\beta)}
  \{\log(1/\tilde\epsilon_n)\}^{1+1/\tau+1/\beta}\nonumber \\
  &\lesssim \tilde\epsilon_n^{2b_1}\{\log(1/\tilde\epsilon_n)\}^{1/2}
\end{align}
since $d(4b_1-1/\beta)>2b_1$ for $b_1>1/(4d-2)\beta$ and the latter follows by the assumption made on $b_1$.

Place disjoint balls $U_j$ with centers at $\mu_1,\ldots,\mu_N$ with diameter $\sigma_n\tilde\epsilon_n^{2 b_1}$ each. 
Because of assumption \eqref{eq:fstar} on $f^*$, for any $\|x\|\leq a_{\sigma_n}$,
  $$f^*(x)\geq c^*\edr^{-b^*a_{\sigma_n}^{\tau^*}}
  \geq c^*\exp\{-b^* a_0^{\tau^*}\log(1/\sigma_n)^{\tau^*/\tau}\}
  \geq c^*\sigma_n^{b^* a_0^{\tau^*}
  \log(1/\sigma_n)^{\tau^*/\tau-1}}
  \gtrsim \sigma_n^{b^* a_0^{\tau^*}},$$
where the second inequality follows from the assumption that $\tau^*\leq\tau$. This implies that for $n$ sufficiently large and some constant $a_1$,
  $a_1 \sigma_n^{b^* a_0^{\tau^*}}(\sigma_n\tilde\epsilon_n^{2 b_1})^d
  \leq P^*(U_j)\leq 1$.
Further extend this to a partition $U_1,\ldots,U_K$ of $\R^d$ such that 
  $a_1 \sigma_n^{b^* a_0^{\tau^*}}(\sigma_n\tilde\epsilon_n^{2 b_1})^d
  \leq P^*(U_j)
  \leq 1$
still holds. We can still have
  $K\lesssim\sigma_n^{-d}\{\log(1/\tilde\epsilon_n)\}^{d+d/\tau}
  = \tilde\epsilon_n^{-d/\beta}
  \{\log(1/\tilde\epsilon_n)\}^{sd}$  
for $s=1+1/\beta+1/\tau$. 
The next step is to construct a partition of $\R^d\times{\cal S}$. To this aim, let 
  $$S_{\sigma_n}
  =\{\Sigma\in{\cal S}:\
  \sigma_n^{-2}< \lambda_i(\Sigma^{-1})
  < \sigma_n^{-2}(1+\sigma_n^{2\beta}),\ i=1,\ldots,d\}$$
which satisfies
  $P^*(S_{\sigma_n})\gtrsim
  \edr^{-C_3\sigma_n^{-\kappa^*}}$
by hypothesis \eqref{eq:Gstar}.
Define $V_j=U_j\times S_{\sigma_n}$ for $j=1,\ldots,K$, and extend $\{ V_1,\ldots,V_K \}$ to a partition $\{V_1,\ldots,V_{M}\}$ of $\R^d\times {\cal S}$ such that
  $$a_2 \sigma_n^{b^* a_0^{\tau^*}+d}
  \sigma_n\tilde\epsilon_n^{2db_1}
  \edr^{-C_3\sigma_n^{-\kappa^*}}
  \leq\alpha P^*(V_j)\leq 1$$
for all $j=1,\ldots,M$, for some positive constant $a_2$ and $n$ sufficiently large. We can still have 
  $M
  \lesssim
  \tilde\epsilon_n^{-d/\beta}
  \{\log(1/\tilde\epsilon_n)\}^{sd}$.   

Set $p_j=0$ for $j>N$ and consider the set ${\cal P}_{\sigma_n}\subset{\mathscr P}$ of probability measures $P$ on $\R^d\times{\cal S}$ with
  $$\sum\nolimits_{j=1}^{M}|P(V_j)-p_j|
  \leq 2\tilde\epsilon_n^{2db_1}\edr^{-C_3\sigma_n^{-\kappa^*}}.$$
Note that
\begin{align*}
  M\;\tilde\epsilon_n^{2db_1}\edr^{-C_3\sigma_n^{-\kappa^*}}
  &\lesssim \tilde\epsilon_n^{-d/\beta}
  \{\log(1/\tilde\epsilon_n)\}^{sd}
  \tilde\epsilon_n^{2db_1}\edr^{-C_3\sigma_n^{-\kappa^*}}\leq 1\\
  \min_{1\leq j\leq M} [\alpha P^*(V_j)]^{1/3}
  &\geq \left\{ a_2 \sigma_n^{b^* a_0^{\tau^*}+d} 
  \tilde\epsilon_n^{2db_1}
  \edr^{-C_3\sigma_n^{-\kappa^*}}\right\}^{1/3}\\
  &=a_2^{1/3}\tilde\epsilon_n^{2db_1}\edr^{-C_3\sigma_n^{-\kappa^*}}
  \left\{ \sigma_n^{b^* a_0^{\tau^*}+d}
  \tilde\epsilon_n^{-4db_1}\edr^{2C_3\sigma_n^{-\kappa^*}}\right\}^{1/3}
  \geq a_2^{1/3}\tilde\epsilon_n^{2db_1}\edr^{-C_3\sigma_n^{-\kappa^*}}
\end{align*}
for $n$ sufficiently large. 
Hence, by Lemma 10 of \citet{ghos:vand:2007},
\begin{align}\label{eq:pp4}
  {\cal D}_{\alpha,P^*}({\cal P}_{\sigma_n})
  &\geq C'\exp\left\{-c\,M\log\left[
  1/\tilde\epsilon_n^{2db_1}\edr^{-C_3\sigma_n^{-\kappa^*}}
  \right]\right\}
  \geq C''\exp\{-c'M\sigma_n^{-\kappa^*}\}\nonumber\\
  &\geq C''\exp\{-c''\tilde\epsilon_n^{-d/\beta}
  \{\log(1/\tilde\epsilon_n)\}^{sd}\sigma_n^{-\kappa^*}\}\nonumber\\
  &=C''\exp\{-c'''\tilde\epsilon_n^{-(d+\kappa^*)/\beta}
  \{\log(1/\tilde\epsilon_n)\}^{d(s+\kappa^*/\beta)}\}
  \geq \exp\{-Cn\tilde\epsilon_n^2\}
\end{align}
for some constants $C,C',C'',c,c',c'',c'''$ and for $t\geq d(s+\kappa^*/\beta)/(2+(d+\kappa^*)/\beta)$ since
\begin{align*}
  \tilde\epsilon_n^{-(d+\kappa^*)/\beta}
  \{\log(1/\tilde\epsilon_n)\}^{d(s+\kappa^*/\beta)}
  &=n^{\frac{d+\kappa^*}{2\beta+d+\kappa^*}}(\log n)^{-(d+\kappa^*)t/\beta}
  \{\log[ n^{\frac{\beta}{2\beta+d+\kappa^*}}
  (\log n)^{-t}] \}^{d(s+\kappa^*/\beta)}\\
  &\lesssim n^{\frac{d+\kappa^*}{2\beta+d+\kappa^*}}
  (\log n)^{d(s+\kappa^*/\beta)-(d+\kappa^*)t/\beta}
  \leq n\tilde\epsilon_n^2
\end{align*}
when $d(s+\kappa^*/\beta)-(d+\kappa^*)t/\beta\leq 2t$, cfr. \eqref{eq:rate_sub}.

We show next that, for any $P\in{\cal P}_{\sigma_n}$, 
  $d(f_0,f_P)
  \lesssim\sigma_n^\beta$.
By triangle inequality,
\begin{align*}
  d(f_0,f_P)
  &\leq 
  d(f_0,\phi_{\sigma_n}\star \tilde h_{\sigma_n})
  +d(\phi_{\sigma_n}\star \tilde h_{\sigma_n},
  \phi_{\sigma_n}\star F_{\sigma_n})
  +d(f_P,\phi_{\sigma_n}\star F_{\sigma_n})  
\end{align*}
and $d(f_0,\phi_{\sigma_n}\star \tilde h_{\sigma_n})\lesssim \sigma_n^{\beta}$ by Proposition 1 of \citet{shen:etal:2012}, while, by \eqref{eq:discr}, 
  $$d(\phi_{\sigma_n}\star \tilde h_{\sigma_n},
  \phi_{\sigma_n}\star F_{\sigma_n})
  \lesssim \tilde\epsilon_n^{b_1}\{\log(1/\tilde\epsilon_n)\}^{1/4}
  \leq \tilde\epsilon_n
  \{\log (1/\tilde\epsilon_n)\}^{-1}
  =\sigma_n^{\beta}$$
since, by assumption, 
  $\tilde\epsilon_n^{b_1}\{\log(1/\epsilon_n)\}^{5/4}
  \leq \tilde\epsilon_n$.
In order to show that also $d(f_P,\phi_{\sigma_n}\star F_{\sigma_n})\lesssim \sigma_n^\beta$, we use arguments presented in Lemma 5 of \citet{ghos:vand:2007}. Let $V_0=\bigcup_{j>N} V_j$. After some algebra,
\begin{multline*}
  f_P(x)-\phi_{\sigma_n}\star F_{\sigma_n}(x)
  =\int_{V_0} \phi_\Sigma(x-\mu)\ddr P(\mu,\Sigma)
  +\sum\nolimits_{j=1}^N
  \int_{V_j} (\phi_\Sigma(x-\mu)-\phi_\Sigma(x-\mu_j)\ddr P(\mu,\Sigma)\\ 
  +\sum\nolimits_{j=1}^N
  \int_{V_j} (\phi_\Sigma(x-\mu_j)-\phi_{\sigma_n}(x-\mu_j)\ddr P(\mu,\Sigma) 
  +\sum\nolimits_{j=1}^N
  \phi_{\sigma_n}(x-\mu_j)\left[ P(V_j)-p_j \right]
\end{multline*}
Note that $P(V_0)=1-\sum_{j=1}^NP(V_j)\leq\sum_{j=1}^N|P(V_j)-p_j|$. Hence, by triangle inequality,
\begin{multline*}
  \|f_P-\phi_{\sigma_n}\star F_{\sigma_n}\|_1
  \leq\sum\nolimits_{j=1}^N
  \int_{V_j} \|\phi_\Sigma(\cdot-\mu)-\phi_\Sigma(\cdot-\mu_j)\|_1
  \ddr P(\mu,\Sigma)\\ 
  +\sum\nolimits_{j=1}^N
  \int_{V_j} \|\phi_\Sigma-\phi_{\sigma_n}\|_1\ddr P(\mu,\Sigma) 
  +2\sum\nolimits_{j=1}^N\left| P(V_j)-p_j \right|
\end{multline*}
Recall now that $V_j=U_j\times S_{\sigma_n}$ for $j=1,\ldots,N$ so that, by construction,
  $\|\mu-\mu'\|\leq\sigma_n\tilde\epsilon_n^{2b_1}$
for any $\mu,\mu'\in U_j$.
Also any $\Sigma\in S_{\sigma_n}$ satisfies $\det(\Sigma^{-1})\geq 2^{-d}\sigma_n^{-2d}$, $y^T\Sigma^{-1}y\leq \|y\|\sigma_n^{-2}$ for any $y\in\R^d$ and  
  $|\tr(\sigma_n^2\Sigma^{-1})-d-\log\det(\sigma_n^2\Sigma^{-1})|
  \leq d\sigma_n^{2\beta}$.
Hence, by using arguments similar to the proof of Lemma \ref{lemma:sieve}, 
\begin{align*}
  &\|\phi_\Sigma(\cdot-\mu)-\phi_\Sigma(\cdot-\mu_j)\|_1
  \leq \|\mu-\mu_j\|\lambda_1(\Sigma^{-1})^{1/2}
  \leq\sigma_n\tilde\epsilon_n^{2b_1}\sigma_n^{-1}
  =\tilde\epsilon_n^{2b_1}\\
  &\|\phi_\Sigma-\phi_{\sigma_n}\|_1
  \leq \left| \tr(\sigma_n^{-2}\Sigma^{-1})
  -d-\log\det(\sigma_n\Sigma^{-1}) \right|
  \leq d\sigma_n^{2\beta}
\end{align*}  
for any $(\mu,\Sigma)\in V_j$ and any $j=1,\ldots,N$. Hence, for $P\in{\cal P}_{\sigma_n}$,
  $$\|f_P-\phi_{\sigma_n}\star F_{\sigma_n}\|_1
  \leq\tilde\epsilon_n^{2b_1} 
  +d\sigma_n^{2\beta}
  +4\tilde\epsilon_n^{2db_1}\edr^{-C_3\sigma_n^{-\kappa^*}}
  \lesssim \sigma_n^{2\beta}$$
since 
  $\tilde\epsilon_n^{2b_1}\leq \sigma_n^{2\beta}$ 
by the assumption on $b_1$. It follows that $d(f_P,\phi_{\sigma_n}\star F_{\sigma_n})\lesssim \sigma_n^\beta$ as desired.

The last step is to control the ratio $f_P/f_0$ for any $P\in{\cal P}_{\sigma_n}$ in view of an application of Lemma B2 of \citet{shen:etal:2012}.  
Since $(\mu_1,\ldots,\mu_N)$ forms a $\sigma_na_{\sigma_n}$-net of $\{x\in\R^d:\ \|x\|\leq a_{\sigma_n}\}$,
the set
  $\{\mu\in\R^d:\ \|\mu-x\|\leq \sigma_n(a_{\sigma_n}+\tilde\epsilon_n^{2b_1}/2)\}$
contains at least one $U_{j}$ for $j=1,\ldots,N$. Call $J(x)$ such index $j$  and recall that, for $P\in{\cal P}_\sigma$,
  $$
  P(V_{J(x)})\geq p_{J(x)}-2\tilde\epsilon_n^{2db_1}
  \edr^{-C_3\sigma_n^{-\kappa^*}}
  \geq \tilde\epsilon_n^{4db_1}-2\tilde\epsilon_n^{2db_1}
  \edr^{-C_3\sigma_n^{-\kappa^*}}
  \gtrsim \tilde\epsilon_n^{4db_1}
  $$ 
Recall that $f_0$ is bounded.
For $\|x\|\leq a_{\sigma_n}$ we have
\begin{align*}
  \frac{f_P(x)}{f_0(x)}
  &\geq K_1 \int\nolimits_{S_{\sigma_n}}
  \int\nolimits_{\|\mu-x\|\leq\sigma_n(a_{\sigma_n}^{\tau/2}
  +\tilde\epsilon_n^{2b_1}/2)}
  \phi_\Sigma(x-z)\ddr P(z,\Sigma)\\
  &\geq K_2 \sigma_n^{-d}
  \exp\left\{
  -\frac{1}{2\sigma_n^2}[\sigma_n(a_{\sigma_n}^{\tau/2}
  +\tilde\epsilon_n^{2b_1}/2)]^2
  \right\}
  P(V_{J(x)})\\
  &\geq K_2\sigma_n^{-d}\edr^{-2a_{\sigma_n}^\tau}P(V_{J(x)})
  \geq K_3\sigma_n^{-d}
  \edr^{-2a_{\sigma_n}^\tau}\tilde\epsilon_n^{4db_1}
\end{align*}
for some constants $K_1,K_2,K_3$. Also, for every $\|x\|>a_{\sigma_n}$,
\begin{align*}
  \frac{f_P(x)}{f_0(x)}
  &\geq K_1 \int\nolimits_{S_{\sigma_n}}
  \int\nolimits_{\|\mu\|\leq a_{\sigma_n}}
  \phi_\Sigma(x-\mu)\ddr P(\mu,\Sigma)\\
  &\geq K_2 \sigma_n^{-d}
  \int\nolimits_{S_{\sigma_n}}
  \int\nolimits_{\|\mu\|\leq a_{\sigma_n}}
  \exp\left\{-\frac{1}{2\sigma^2}\|x-\mu\|^2\right\}\ddr P(\mu,\sigma)\\
  &\geq K_2 \sigma_n^{-d}
  \exp\{-2\|x\|^2/\sigma_n^2\}
  P\left(\{\mu\in\R: \|\mu\|\leq a_{\sigma_n}\}\times S_{\sigma_n}\right)\\
  &\geq K_2\sigma_n^{-d}\exp\{-2\|x\|^2/\sigma_n^2\}
  \sum\nolimits_{j=1}^N P(V_j)
  \geq K_4\sigma_n^{-d}\exp\{-2\|x\|^2/\sigma_n^2\}
\end{align*}
for some constant $K_4$ because $\|x-\mu\|^2\leq 2\|x\|^2+2\|\mu\|^2\leq 4\|x\|^2$ (since $\|\mu\|\leq a_{\sigma_n}<\|x\|$) and 
\begin{align*}
  \sum\nolimits_{j=1}^N P(V_j)
  &=1-\sum\nolimits_{j>N}^{M} P(V_j)
  =1-\sum\nolimits_{j>N}^{M} |P(V_j)-p_j|\\
  &\geq 1-\sum\nolimits_{j=1}^{M} |P(V_j)-p_j|
  \geq 1-2\tilde\epsilon_n^{2db_1}\edr^{-C_3\sigma_n^{-\kappa^*}}
\end{align*}
and the last inequality follows from the definition of ${\cal P}_\sigma$. 
Set
  $\lambda=K_3\sigma_n^{-d}\edr^{-2a_{\sigma_n}^\tau}
  \tilde\epsilon_n^{4db_1}$ 
and notice that 
\begin{align*}
  \log(1/\lambda)
  &=\log(1/\tilde\epsilon_n^{4db_1}\sigma_n^{-d})
  +\log\left(\exp\{2a_0^2\log(1/\sigma_n)\}
  \right)\\
  &=\log\left(1/\tilde\epsilon_n^{d(4b_1-1/\beta)}
  \{\log(1/\tilde\epsilon_n)\}^{d/\beta}\right)
  +2a_0^2\log(1/\sigma_n)
  \lesssim \log(1/\tilde\epsilon_n)
\end{align*}
since $4b_1>1/\beta$. 
Moreover, for any $P\in{\cal P}_\sigma$,
  $\{x\in\R: f_P(x)/f_0(x)<\lambda\}
  \subset
  \{x\in\R: \|x\|>a_{\sigma_n}\}$
so that
\begin{align*}
  F_0\left\{\left(\log\frac{f_0}{f_P}\right)^2\Indic
  \left(\frac{f_P}{f_0}<\lambda\right)\right\}
  &\leq \int\nolimits_{\|x\|>a_{\sigma_n}}
  \left(\log\frac{f_P(x)}{f_0(x)}\right)^2f_0(x)\ddr x\\
  &\leq \int\nolimits_{\|x\|>a_{\sigma_n}}
  \left(\log\left[
  \frac{\sigma_n^d}{K_4}\exp\{-2\|x\|^2/\sigma_n^2\}
  \right]\right)^2
  f_0(x)\ddr x\\
  &\leq \frac{K_5}{\sigma_n^4}
  \int\nolimits_{\|x\|>a_{\sigma_n}}
  \|x\|^4f_0(x)\ddr x  
  = \frac{K_5}{\sigma_n^4}
  \int\nolimits_{\|x\|>a_{\sigma_n}}
  \|x\|^4f_0(x)^{1/2}\, f_0(x)^{1/2}\ddr x\\
  &\leq \frac{K_5}{\sigma_n^4}
  \left\{\int\nolimits_{\|x\|>a_{\sigma_n}}\|x\|^8f_0(x)\ddr x
  \int\nolimits_{\|x\|>a_{\sigma_n}} f_0(x)\ddr x \right\}^{1/2}\\
  &\leq \frac{K_5}{\sigma_n^4}
  (F_0 \|X\|^8)^{1/2}(F_0\{x: \|x\|>a_{\sigma_n}\})^{1/2}
  \leq \frac{K_6}{\sigma_n^4}
  (F_0\{x: \|x\|>a_{\sigma_n}\})^{1/2}\\
  &\leq \frac{K_6}{\sigma_n^4}F_0(E_{\sigma_n}^c)^{1/2}
  \leq \frac{K_6}{\sigma_n^4} B_0 \sigma_n^{2\beta+\epsilon+4}
  \leq K_7\sigma_n^{2\beta+\epsilon}
\end{align*}
for some constants $K_5,K_6,K_7$. The forth inequality follows from Cauchy-Schwartz, the sixth inequality follows from $F_0 \|X\|^m<\infty$ for all $m>0$ because of the tail condition \eqref{eq:8} on $f_0$, the seventh and eighth inequalities follows from Proposition 1 of \citet{shen:etal:2012}. Since $\log x\leq (\log x)^2$ for $x>\edr^1$ and $\lambda<\edr^{-1}$ for $n$ sufficiently large, we also have that
  $F_0\{\log(f_0/f_P)\Indic(f_P/f_0<\lambda)\}
  \leq K_8\sigma_n^{2\beta+\epsilon}$.
Now apply Lemma B2 of \citet{shen:etal:2012} to conclude that both $F_0\{\log f_0/f_P\}$ and $F_0\{(\log f_0/f_P)^2\}$ are bounded by 
  $$K_9\{\log(1/\lambda)\}^2\sigma_n^{2\beta}
  \leq A \sigma_n^{2\beta}\{\log(1/\tilde\epsilon_n)\}^2
  =A \tilde\epsilon_n^2$$
for some positive constant $A$. Together with \eqref{eq:pp4}, this completes the proof.
\end{proof}

\bibliographystyle{asa}
\bibliography{biblio}

\end{document}